\documentclass[letterpaper]{article}
\usepackage[T1]{fontenc}
\usepackage[english]{babel}
\usepackage{color,xcolor}
\usepackage{fullpage,graphicx,tikz,pgfplots}
\usepackage{amsmath,amsfonts,amssymb,amsthm,mathtools,dsfont}
\usepackage{booktabs,adjustbox}
\usepackage{cite}
\usepackage{enumitem}
\usepackage{hyperref}
\usepackage[capitalize,nameinlink,noabbrev]{cleveref}

\allowdisplaybreaks
\mathtoolsset{centercolon}
\pgfplotsset{compat=1.18}
\usepgfplotslibrary{groupplots}
\usetikzlibrary{calc,positioning,decorations.pathreplacing,arrows.meta}
\hypersetup{colorlinks, hypertexnames=false, pageanchor=true,
            linkcolor=blue, citecolor={green!50!black}, urlcolor=cyan,
            pdftitle={Local linear convergence of PDHG for SDP}
            }

\numberwithin{equation}{section}
\theoremstyle{plain}
\newtheorem{theorem}{{Theorem}}[section]
\newtheorem{lemma}[theorem]{{Lemma}}
\newtheorem{corollary}[theorem]{{Corollary}}
\newtheorem{proposition}[theorem]{{Proposition}}
\newtheorem{assumption}[theorem]{{Assumption}}
\crefname{assumption}{Assumption}{Assumptions}

\theoremstyle{definition}
\newtheorem{definition}[theorem]{{Definition}}
\newtheorem{example}[theorem]{{Example}}

\theoremstyle{remark}

\newcommand{\reals}                  {\mathbb R}
\newcommand{\natint}                 {\mathbb N}
\newcommand{\symm}                   {\mathbb S}
\newcommand{\R}[1][]                 {\reals^{#1}}
\newcommand{\SV}[2]                  {\symm^{#1}_{#2}}
\newcommand{\ones}                   {\mathds 1}

\newcommand{\tran}                   {{\mathsf T}}                
\newcommand{\proj}                   {\Pi}                        
\newcommand{\inprod}[2]              {\langle #1, #2 \rangle}     

\DeclareMathOperator{\Nullspace}     {Null}                       
\DeclareMathOperator{\Range}         {Range}                      
\DeclareMathOperator{\tr}            {{\bf tr}}                   
\DeclareMathOperator{\diag}          {{\bf diag}}                 
\DeclareMathOperator{\dist}          {{\bf dist}}                 
\DeclareMathOperator{\rank}          {rank}                       

\newcommand{\argmin}                 {\mathop{\mathrm{argmin}}}
\newcommand{\mini}                   {\text{\upshape minimize}}
\newcommand{\maxi}                   {\text{\upshape maximize}}
\newcommand{\st}                     {\text{\upshape subject~to}}

\DeclareMathOperator{\aff}           {{\bf aff}}                  

\DeclareMathOperator{\D}             {D}                          

\newcommand{\eg}{{\textit{e.g.}}}
\newcommand{\ie}{{\textit{i.e.}}}

\newcommand{\cA}{\mathcal{A}}
\newcommand{\cB}{\mathcal{B}}
\newcommand{\cC}{\mathcal{C}}
\newcommand{\cD}{\mathcal{D}}
\newcommand{\cL}{\mathcal{L}}
\newcommand{\cM}{\mathcal{M}}
\newcommand{\cN}{\mathcal{N}}
\newcommand{\cO}{\mathcal{O}}
\newcommand{\cT}{\mathcal{T}}
\newcommand{\cU}{\mathcal{U}}
\newcommand{\cV}{\mathcal{V}}

\newcommand{\oline}[1]{\mkern 1.5mu\overline{\mkern-1.5mu#1}}
\newcommand{\Sbar}{\oline{S}}
\newcommand{\Xbar}{\oline{X}}
\newcommand{\Zbar}{\oline{Z}}
\newcommand{\ybar}{\oline{y}}
\newcommand{\zbar}{\oline{z}}
\newcommand{\kbar}{\oline{k}}
\newcommand{\tnorm}[1]{{\vert\kern-0.25ex\vert\kern-0.25ex\vert#1\vert\kern-0.25ex\vert\kern-0.25ex\vert}}

\newcommand{\Ftilde}{\widetilde{F}}

\newcommand{\ztilde}{\tilde{z}}

\newcommand{\Shat}{\widehat{S}}
\newcommand{\zhat}{\hat{z}}

\DeclareMathOperator{\fix}{Fix}
\DeclareMathOperator{\inertia}{inertia}
\newcommand{\fro}{\mathsf{F}}
\newcommand{\Id}{\mathrm{Id}}

\newcommand{\cZ}{\mathcal{Z}}

\newcommand{\cQ}{\mathcal{Q}}
\newcommand{\KKT}{\mathrm{KKT}}
\newcommand{\Pihat}{\widehat{\Pi}}

\DeclareMathOperator{\ri}{ri}
\newcommand{\ID}{{\scriptscriptstyle \mathrm{ID}}}

\title{Local linear convergence of the primal--dual hybrid gradient
algorithm for semidefinite programming}
\author{Xin Jiang%
\thanks{Department of Industrial and Systems Engineering,
University of Houston. Email: \texttt{xinjiang@uh.edu}}
}
\date{November 16, 2026}

\begin{document}
\maketitle

\begin{abstract}
Primal--dual first-order methods are widely used for large-scale
semidefinite programming (SDP), but their ability to compute highly
accurate solutions is not explained by global convergence theory alone.
We study the local convergence of the primal--dual hybrid gradient (PDHG)
method for SDP.  We first show that PDHG converges linearly,
uniformly for all starting points near a KKT point,
if and only if the KKT mapping is metrically subregular at that point.
We then show that PDHG eventually converges linearly whenever the limit
KKT point is strictly complementary, even with multiple solutions.
The asymptotic rate admits a bound computable from the eigenvalues of
the linearized PDHG map and matching the observed rates in our experiments.
As a byproduct of the analysis,
strict complementarity implies metric subregularity.
Finally, we prove that PDHG can converge sublinearly even when started
arbitrarily close to a solution, which cannot happen for linear programming.
Numerical experiments support the theory and exhibit SDP instances on which
PDHG struggles to reach high accuracy.
\end{abstract}

\section{Introduction}

Semidefinite programming is a central model in convex optimization.
It combines affine constraints with the geometry of the positive
semidefinite (PSD) cone, and arises in control, signal processing,
statistics, and many other fields \cite{BV04,Ali95,MHA20}.
The growing demand for solving large-scale instances has motivated the use
of primal--dual first-order methods, whose main operations
are applications of linear maps and projections onto the PSD cone.
This paper studies the local convergence of one such method,
the primal--dual hybrid gradient (PDHG) algorithm,
applied to semidefinite programming.

We consider the primal--dual pair of semidefinite programs (SDPs)
\begin{equation} \label{eq:sdp}
\begin{array}{lll}
  \text{Primal:} & \mini & \inprod{C}{X} \\
  & \st & \cA X=b \\
  & & X\in\SV{n}{+}
\end{array} \qquad \quad \begin{array}{lll}
  \text{Dual:} & \maxi & b^\tran y \\
  & \st & \cA^\ast y+S=C \\
  & & S\in\SV{n}{+}
\end{array}
\end{equation}
with primal variable $X \in \symm^n$ and dual variables $y \in \R[m]$ and
$S \in \symm^n$.  Here $\symm^n$ is the set of real symmetric $n \times n$
matrices, and $\SV{n}{+}$ is the set of positive semidefinite matrices
in $\symm^n$.
The problem data are $C,A_1,\ldots,A_m \in \symm^n$ and $b \in \R[m]$.
The linear operator $\cA \colon \symm^n \to \R[m]$ is defined by
$\cA X = (\inprod{A_1}{X},\ldots,\inprod{A_m}{X})$,
and $\cA^\ast y = \sum_{i=1}^m y_iA_i$ is its adjoint operator.

When applied to \eqref{eq:sdp}, PDHG \cite{EZC10,CP11a} with stepsizes
$\tau,\sigma > 0$ takes the form
\begin{subequations} \label{eq:pdhg}
\begin{align}
  X_{k+1} &= \proj_{\SV{n}{+}}
    \big(X_k - \tau (C - \cA^\ast y_k) \big) \label{eq:pdhg-X} \\
  y_{k+1} &= y_k + \sigma
    \big(b - \cA(2X_{k+1}-X_k) \big). \label{eq:pdhg-y}
\end{align}
\end{subequations}
Unless stated otherwise, the stepsizes are fixed and satisfy
\begin{equation}\label{eq:stepsize}
  \tau > 0, \qquad \sigma > 0, \qquad \tau \sigma \|\cA\|^2 < 1,
\end{equation}
where $\|\cA\|$ is the operator norm induced by the Frobenius norm
on $\symm^n$ and the Euclidean norm on~$\R[m]$.
Each iteration consists of one projection onto the PSD cone
and one application of each of $\cA$ and $\cA^\ast$.
This simple structure makes PDHG attractive for large-scale SDPs
\cite{WLY24}.

Primal--dual first-order methods are often regarded as effective for
moderate accuracy but less reliable when high accuracy is required,
a view consistent with their sublinear global complexity bounds
\cite{CP11a,CP16b}.
For linear programming (LP), however, PDHG-based solvers such as PDLP
compute highly accurate solutions in practice \cite{ADH+21,AHLL23},
and PDHG provably enters a local linear regime, with a rate governed by
the geometry of the optimal face \cite{LY24,XF26}.
The explanation is local: PDHG is a preconditioned proximal point method
for the KKT inclusion \cite{HY12b}, proximal point methods converge
linearly near a solution at which the operator is metrically subregular
\cite{Rockafellar76b,Luq84,Lev09}, and the KKT mapping of an LP is
polyhedral and hence metrically subregular \cite{Rob81}.
For SDP, the PSD cone is not polyhedral, and metric subregularity can
fail (see \cref{sec:sublinear} for such an example).
Local linear convergence of the alternating direction method of
multipliers (ADMM), an algorithm closely related to PDHG \cite{OV20}, has
nevertheless been established under metric subregularity of the KKT mapping
\cite{HSZ18} and under strict complementarity or nondegeneracy \cite{KJY26}.
These results rest on two different arguments.
Under metric subregularity, which is hard to verify directly but is
implied by primal--dual nondegeneracy \cite{CS08,DR04}, the classical
error-bound argument applies.
Under strict complementarity, which can be assessed numerically from the
eigenvalues of an approximate solution, \cite{KJY26} instead uses a
linearization argument at the limit, in which metric subregularity does
not enter.
This paper develops both arguments for PDHG and makes precise
when PDHG for SDP enters a local linear regime, at what rate, and
whether it can converge only sublinearly.

\paragraph{Contributions.}
The theoretical contributions of this paper are threefold.
\begin{itemize}
\item \textit{Metric subregularity and uniform linear convergence.}
We show that PDHG converges linearly near a KKT point,
uniformly in the starting point,
if and only if the KKT mapping is metrically subregular at that point.
Metric subregularity is a property of the SDP, not of the
algorithm, and the characterization holds for all stepsizes
with $\tau \sigma \|\cA\|^2 \le 1$.
Hence whether this linear regime exists does not depend on $\tau$
and~$\sigma$.
Since primal--dual nondegeneracy implies metric
subregularity, it guarantees this linear regime.

\item \textit{Strict complementarity and the asymptotic rate.}
If PDHG converges to a strictly complementary KKT point, then it eventually
converges linearly, even when the solution set is not a singleton.
The proof linearizes PDHG at the projection of each iterate
onto the KKT set, which is locally affine under strict complementarity.
This yields a bound~$\rho_\star$ on the asymptotic rate
that is computable from the eigenvalues of the linearized PDHG map.
In our experiments,~$\rho_\star$ closely matches the observed rates.
As a corollary, strict complementarity implies metric subregularity,
a result recently proved by a different argument \cite{DLY26}.

\item \textit{A provably sublinear example.}
We construct a small SDP instance on which PDHG can converge at the rate
$\Theta(k^{-1/2})$ in the primal variable and $\Theta(k^{-1})$
in the dual variable, even when started arbitrarily close to the solution.
This cannot happen for LP, whose KKT mapping is polyhedral and hence
metrically subregular \cite{Rob81}.
\end{itemize}
Numerical experiments illustrate these results and exhibit instances,
some close to failing strict complementarity,
on which PDHG struggles to reach high accuracy.

\paragraph{Outline.}
\Cref{sec:related} reviews related work.  \Cref{sec:setting} introduces
the problem setting and preliminaries.
\Cref{sec:msr} characterizes uniform local linear convergence
using metric subregularity and discusses the degenerate case
$\tau \sigma \|\cA\|^2 = 1$ and primal--dual nondegeneracy.
\Cref{sec:sc} analyzes local linear convergence of PDHG
under strict complementarity.  \Cref{sec:sublinear} presents an SDP instance
for which PDHG converges at a sublinear rate.
\Cref{sec:rank-id} discusses the rank identification phenomenon.
\Cref{sec:numerics} reports numerical experiments,
and \cref{sec:conclusion} concludes the paper.

\section{Related work} \label{sec:related}

\paragraph{Algorithms for semidefinite programming.}
Semidefinite programs have traditionally been solved to high accuracy by
interior-point methods \cite{VB96,Todd01,WSV00,Sturm99,TTT02},
whose per-iteration cost limits their use on very large instances.
This has motivated techniques that exploit chordal sparsity
\cite{VA15,JV22}, augmented-Lagrangian, ADMM, and semismooth Newton methods
\cite{WGY10,ZST10,YST15,GKY25}, low-rank factorization methods
\cite{BM03,BM05}, and more recent low-rank splitting schemes
\cite{DLY25,HLL+25,MSC26,ACG+25,TT26b}.
The present paper does not design a new SDP solver; instead, it studies
the local behavior of PDHG applied directly to the pair~\eqref{eq:sdp}.

\paragraph{PDHG.}
The PDHG algorithm, also known as the Chambolle--Pock method,
is a primal--dual splitting method for saddle-point problems
and monotone inclusions \cite{EZC10,CP11a,CP16b,PCBC09}.
When $\tau\sigma\|\cA\|^2<1$, PDHG is a preconditioned proximal point method
for the KKT inclusion \cite{HY12b,JV23,OV20}; when $\tau\sigma\|\cA\|^2=1$,
the preconditioner is singular, but PDHG still converges \cite{BCLN22}.
For LP, PDHG has become a practical large-scale method through the PDLP
line of work \cite{ADH+21,ADH+26,AHLL23,LY25,CSY+26}, which has
stimulated refined analyses such as \cite{LY24,XF26,Xiong26};
it has since been extended to other conic programs \cite{LXGY25}
and adapted to large-scale SDPs \cite{WLY24}.

\paragraph{Error bounds and linear convergence.}
For the proximal point method, linear convergence follows from
Lipschitz-type error bounds on the inverse operator at zero
\cite{Rockafellar76b,Luq84} and, more generally, from metric subregularity
at a solution \cite{Lev09}.
For general fixed-point iterations, error bounds on the fixed-point
residual are also necessary when linear convergence holds uniformly near a
fixed point \cite{LTT20}; for averaged operators they are necessary and
sufficient, and they hold for piecewise linear operators~\cite{vTPVZ25}.
For Douglas--Rachford splitting and ADMM \cite{GM76,EB92,BPC+11}, linear
rates follow from strong convexity and smoothness \cite{DY16,GB17} or from
metric subregularity of the KKT mapping \cite{HSZ18}.
We compare these results with ours after \cref{thm:msr}.

\paragraph{Local analyses of ADMM and PDHG for SDP.}
A recent line of work refines the global theory of primal--dual splitting
methods for SDP by analyzing their behavior near a solution.
Han, Sun, and Zhang \cite{HSZ18} proved linear convergence of ADMM under
metric subregularity, which holds under primal--dual nondegeneracy
\cite{CS08,DR04}.
Kang, Jiang, and Yang \cite{KJY26} proved local linear convergence of ADMM
under either strict complementarity or nondegeneracy;
see \cref{sec:admm} for a detailed discussion.
Kang and Yang \cite{KY26} studied the second-order limit dynamics of ADMM.
Ding, Lu, and Yang \cite{DLY26} showed for conic programs that
strict complementarity together with a quadratic facial-violation condition,
which holds for the PSD cone, implies metric subregularity
and hence local linear convergence of PDHG and ADMM.

\paragraph{Partial smoothness.}
Under a regularity condition, which in our setting is equivalent to strict
complementarity, proximal-type methods whose iterates converge to a solution
identify the active manifold of a partly smooth function after finitely
many iterations \cite{Lewis02,HL04,DL14,Wright93}.  For SDP, this manifold
consists of the PSD matrices whose rank equals that of the solution.
Identification was combined with a local linearization for
forward--backward \cite{LFP17a}, Douglas--Rachford and ADMM \cite{LFP17b},
and primal--dual splitting \cite{LFP18}.
\Cref{sec:lfp} relates our analysis to \cite{LFP18}.

\section{Preliminaries} \label{sec:setting}

This section presents the two ingredients in our analysis.
The first is algorithmic: PDHG is a preconditioned proximal point method
for the KKT mapping of \eqref{eq:sdp} (\cref{sec:kkt,sec:ppm}).
The second is variational: a regularity property of the KKT mapping
at the limit point (\cref{sec:regularity}).

\subsection{Notation} \label{sec:notation}

Let $\symm^n$ denote the set of real symmetric $n \times n$ matrices,
equipped with the inner product $\inprod{X}{Y}=\tr(XY)$
and the Frobenius norm $\|X\|_\fro$.  The cone of positive semidefinite
matrices in $\symm^n$ is denoted by $\SV{n}{+}$.
We write $X\succeq 0$ and $X\succ 0$ to mean that $X$ is
positive semidefinite and positive definite, respectively.
We write $\|\cdot\|_2$ for the Euclidean norm of a vector and the spectral
norm of a matrix, and $\|\cdot\|$ for the Euclidean norm on product spaces
such as $\symm^n\times\R[m]$, induced by the Frobenius and Euclidean norms;
for a linear map $\cA$ between such spaces, $\|\cA\|$ is its operator norm.
For a linear operator $A$, its range and nullspace are denoted by
$\Range(A)$ and $\Nullspace(A)$, respectively.
The rank and inertia of $X \in \symm^n$ are denoted by
$\rank X$ and $\inertia(X)$, respectively.
For a mapping $T$, its fixed-point set is $\fix(T) = \{u\mid T(u) = u\}$.
Let $e_i$ denote the $i$th unit vector. Define $E_{ii}=e_ie_i^\tran$
and $E_{ij} = (e_ie_j^\tran+e_je_i^\tran)/2$,
where the ambient dimension is understood from context.
We write $\natint=\{0,1,2,\ldots\}$, $[n]=\{1,\ldots,n\}$,
$\diag(x)$ for the diagonal matrix with diagonal $x$, and $\circ$ for
the Hadamard product.  For a linear subspace $\cL$, denote by $\cL^\perp$
its orthogonal complement.

For a set $\cC$, $\aff \cC$ denotes its affine hull and
$\dist(x,\cC)=\inf_{y\in\cC}\|x-y\|$ the Euclidean distance from $x$ to
$\cC$, with $\dist(x,\emptyset)=+\infty$.
If $\cC$ is nonempty, closed, and convex, then $\proj_\cC$ denotes the
Euclidean projection onto $\cC$, and the normal cone to $\cC$ at $x$ is
$N_\cC (x) = \{v \mid \inprod{v}{y-x} \le 0 \text{ for all } y \in \cC\}$
if $x \in \cC$ and $N_\cC(x)=\emptyset$ otherwise.
The closed Euclidean ball of radius $\epsilon$ centered at $x$
is denoted by $\cB_\epsilon(x)$.

For a self-adjoint positive semidefinite operator $P$, define
$\inprod{u}{v}_P = \inprod{u}{Pv}$ and $\|u\|_P = \sqrt{\inprod{u}{u}_P}$;
this is a seminorm, and a norm if $P$ is positive definite.
We write $u \perp_P v$ if $\inprod{u}{v}_P = 0$, and
$\dist_P (u,\cC) = \inf_{v\in \cC} \|u-v\|_P$ for a set~$\cC$,
with $\dist_P(u,\emptyset)=+\infty$.
If $P$ is positive definite and $\cC$ is nonempty, closed, and convex, then
$\proj_\cC^P (u) = \argmin_{v\in \cC} \|u-v\|_P$ denotes the projection
in the $P$-metric.
We denote by $\cB_\epsilon^P(u)$ the closed
$P$-metric ball of radius $\epsilon$ centered at $u$.

\subsection{Problem setting} \label{sec:kkt}

We make the following assumption on \eqref{eq:sdp}.
\begin{assumption} \label{asp:prob}
The linear map $\cA$ is surjective, and the KKT set
\[
\Omega_\star = \{(X,y,S) \in \symm^n \times \R[m] \times \symm^n \mid
\cA X=b, \, X\in\SV{n}{+}, \, \cA^\ast y+S=C, \, S\in\SV{n}{+}, \,
\inprod{X}{S}=0 \}
\]
is nonempty.
\end{assumption}
Since $X_\star,S_\star\succeq0$ and $\inprod{X_\star}{S_\star}=0$, any KKT
point $(X_\star,y_\star,S_\star)$ satisfies $X_\star S_\star=0$.
Hence $X_\star$ and $S_\star$ admit a simultaneous orthogonal decomposition
\begin{equation} \label{eq:eigen}
\begin{array}{ll}
  X_\star = Q_\star \begin{bmatrix}
    \Lambda_X & 0 \\ 0 & 0
  \end{bmatrix} Q_\star^\tran, \qquad &
  \Lambda_X = \diag(\lambda_1,\ldots,\lambda_r), \\[2ex]
  S_\star = Q_\star \begin{bmatrix}
    0 & 0 \\ 0 & \Lambda_S
  \end{bmatrix} Q_\star^\tran, &
  \Lambda_S = \diag(\mu_1,\ldots,\mu_s),
\end{array}
\end{equation}
where $Q_\star \in \R[n \times n]$ is orthogonal,
$\lambda_1 \ge \cdots \ge \lambda_r > 0$, and $\mu_1\ge\cdots\ge\mu_s>0$.
Thus $r=\rank X_\star$, $s=\rank S_\star$, and $r+s\le n$.

We use two KKT mappings.  The extended KKT mapping is
\begin{equation} \label{eq:kkt-ext}
  \Ftilde_\KKT(X,y,S) = \begin{bmatrix}
	\cA X - b \\ \cA^\ast y + S - C \\ S + N_{\SV{n}{+}} (X)
  \end{bmatrix},
\end{equation}
so that $0 \in \Ftilde_\KKT(X,y,S)$ if and only if
$(X,y,S)\in\Omega_\star$.  Eliminating $S = C - \cA^\ast y$ gives the
reduced KKT mapping on $\symm^n\times\R[m]$,
\begin{equation} \label{eq:kkt}
  F_\KKT(X,y) = \begin{bmatrix}
	C - \cA^\ast y + N_{\SV{n}{+}} (X) \\ \cA X - b
  \end{bmatrix},
\end{equation}
whose zero set is $\cZ_\star = F_\KKT^{-1}(0) =
\{(X,y)\mid (X,y,C-\cA^\ast y) \in \Omega_\star\}$.
The mapping $F_\KKT$ is maximal monotone: it is the sum of the normal cone
mapping of $\SV{n}{+}\times\R[m]$, a skew-adjoint linear mapping, and a
constant \cite[Corollary~25.5]{BC17}.
Throughout, $z=(X,y)$ denotes a primal--dual pair, and we write
$z_\star=(X_\star,y_\star)$ for a point of $\cZ_\star$.

\subsection{PDHG as a preconditioned proximal point method} \label{sec:ppm}

It is known that PDHG (with suitable stepsizes) can be viewed as
a preconditioned proximal point method for $F_\KKT$ \cite{HY12b,JV23,OV20}.
Define the self-adjoint block operator on $\symm^n \times \R[m]$
\begin{equation} \label{eq:P}
  P = \begin{bmatrix}
	\tau^{-1} \Id & \cA^\ast \\ \cA & \sigma^{-1} I_m
  \end{bmatrix},
\end{equation}
and the associated bilinear form and seminorm
\begin{equation} \label{eq:P-metric}
\begin{aligned}
  \inprod{(U,v)}{(W,w)}_P
	&= \inprod{U}{\tau^{-1}W+\cA^\ast w} + v^\tran (\cA W+\sigma^{-1}w), \\
  \|(U,v)\|_P^2
    &= \tau^{-1} \|U\|_\fro^2 + 2v^\tran \cA U + \sigma^{-1} \|v\|_2^2.
\end{aligned}
\end{equation}
By a Schur complement argument, $P \succeq 0$ if and only if
$\tau\sigma\|\cA\|^2\le 1$, and $P\succ 0$ if and only if
$\tau\sigma\|\cA\|^2<1$.

For $z=(X,y)$ and $r=(r_X,r_y)$ in $\symm^n\times\R[m]$, consider the
inclusion
\begin{equation} \label{eq:ppm-r}
  r \in F_\KKT(z_+) + P(z_+ - z).
\end{equation}
Its first block reads
$r_X\in C-\cA^\ast y+N_{\SV{n}{+}}(X_+)+\tau^{-1}(X_+-X)$, since the terms
in $y_+$ cancel.  Hence \eqref{eq:ppm-r} has the unique solution
\begin{equation} \label{eq:step}
  X_+=\proj_{\SV{n}{+}}\big(X-\tau(C-\cA^\ast y-r_X)\big), \qquad
  y_+=y+\sigma\big(b-\cA(2X_+-X)+r_y\big).
\end{equation}
In particular, $z_+$ depends on $(z,r)$ only through $Pz+r$, and, since the
projection is nonexpansive, there is a constant $L>0$, depending only on
$\tau$, $\sigma$, and $\cA$, such that the solutions $z_+$ and $z_+'$ of
\eqref{eq:ppm-r} for $(z,r)$ and $(z',r')$ satisfy
\begin{equation} \label{eq:step-lip}
  \|z_+-z_+'\|\le L\,\|P(z-z')+r-r'\|.
\end{equation}
For $r=0$, \eqref{eq:step} is the PDHG step \eqref{eq:pdhg}, and we write
\begin{equation} \label{eq:pdhg-map}
  z_{k+1} = R z_k, \quad \text{where } R = (P+F_\KKT)^{-1} P.
\end{equation}
Thus PDHG is the preconditioned proximal point method
\begin{equation} \label{eq:ppm}
  0 \in F_\KKT (z_{k+1}) + P (z_{k+1} - z_k),
\end{equation}
and $\fix(R)=\cZ_\star$ \cite{HY12b,JV23,OV20}.
If $\tau \sigma \|\cA\|^2 \le 1$, the map $R$ is firmly nonexpansive
in the $P$-seminorm
\begin{equation} \label{eq:fne}
  \|Rz - Rz'\|_P^2 \le \inprod{Rz-Rz'}{z-z'}_P \quad \text{for all } z,z',
\end{equation}
which follows from monotonicity of $F_\KKT$, applied to the inclusions
$P(z-Rz)\in F_\KKT(Rz)$ and $P(z'-Rz')\in F_\KKT(Rz')$
\cite[Proposition~23.8]{BC17}, \cite{BCLN22}.
Then, taking $z' \gets \zbar \in \cZ_\star$ gives
\begin{equation} \label{eq:fejer}
  \|Rz-\zbar\|_P^2 + \|z-Rz\|_P^2 \le \|z-\zbar\|_P^2
  \qquad \text{for all } z \text{ and all } \zbar\in\cZ_\star,
\end{equation}
and, taking the infimum over $\zbar\in\cZ_\star$,
\begin{equation} \label{eq:fejer-dist}
  \dist_P(Rz,\cZ_\star)^2 + \|z-Rz\|_P^2 \le \dist_P(z,\cZ_\star)^2
  \qquad \text{for all } z.
\end{equation}
For $\tau\sigma\|\cA\|^2\le1$, every PDHG sequence converges to a point of
$\cZ_\star$ \cite{CP11a,BCLN22,BUG26}.

\subsection{Regularity conditions} \label{sec:regularity}

We review the three regularity conditions at a KKT point; their relations,
some of which are established in \cref{sec:msr,sec:sc},
are summarized in \cref{fig:implications}.

\paragraph{Strict complementarity.}
A KKT point $(X_\star,y_\star,S_\star)$ is \textit{strictly complementary}
if $X_\star + S_\star \succ 0$, or equivalently in \eqref{eq:eigen},
\begin{equation} \label{eq:sc}
  \rank X_\star + \rank S_\star = n.
\end{equation}
The existence of a strictly complementary solution is a generic property
of SDPs \cite[Theorem~15]{AHO97}.
Under strict complementarity, $X_\star-\tau S_\star$ is nonsingular,
so $\proj_{\SV{n}{+}}$ is smooth near it (\cref{sec:sc}).
The convergence results in \cref{sec:sc} use the following assumption.
\begin{assumption} \label{asp:sc}
The PDHG iteration \eqref{eq:pdhg} converges to a KKT point
satisfying strict complementarity~\eqref{eq:sc}.
\end{assumption}

\paragraph{Primal--dual nondegeneracy.}
Primal and dual nondegeneracy are defined through four subspaces
\cite{AHO97}:
\begin{align*}
  \cT_{X_\star} &= \left\{Q_\star \begin{bmatrix}
      H_{11} & H_{21}^\tran \\ H_{21} & 0
	\end{bmatrix} Q_\star^\tran \;\middle|\; H_{11} \in \symm^r, \,
	H_{21} \in \R[(n-r) \times r] \right\}, &
  \cN_{X_\star} &= \cT_{X_\star}^\perp, \\
  \cT_{S_\star} &= \left\{Q_\star \begin{bmatrix}
	  0 & H_{21}^\tran \\ H_{21} & H_{22}
	\end{bmatrix} Q_\star^\tran \;\middle|\; H_{22} \in \symm^s, \,
	H_{21} \in \R[s \times (n-s)] \right\}, &
  \cN_{S_\star} &= \cT_{S_\star}^\perp.
\end{align*}
Here $\cT_{X_\star}$ and $\cT_{S_\star}$ are the lineality spaces of
the tangent cones to $\SV{n}{+}$ at $X_\star$ and $S_\star$, respectively.
A KKT point $(X_\star,y_\star,S_\star)$ satisfies
\textit{primal--dual nondegeneracy} (ND) if
\begin{equation} \label{eq:pdnd}
\begin{array}{llcl}
  \text{Primal ND:} &
	\cT_{X_\star} + \Nullspace (\cA) = \symm^n & \Longleftrightarrow &
	\cN_{X_\star} \cap \Range (\cA^\ast) = \{0\}, \\
  \text{Dual ND:} &
    \cT_{S_\star} + \Range (\cA^\ast) = \symm^n & \Longleftrightarrow &
	\cN_{S_\star} \cap \Nullspace (\cA) = \{0\}.
\end{array}
\end{equation}
Primal (resp., dual) nondegeneracy implies uniqueness of the dual (resp.,
primal) solution; under strict complementarity the converse also holds
\cite{AHO97}.

\paragraph{Metric subregularity.}
The following notion is central to \cref{sec:msr}; see
\cite{DR04,DR09,Lev09}.
\begin{definition} \label{def:msr}
Let $F$ be a set-valued mapping between two Euclidean spaces and
$\zbar \in F^{-1}(0)$.  The mapping $F$ is \textit{metrically subregular}
at $\zbar$ for $0$ if there exist $\kappa>0$ and $\epsilon>0$ such that
\[
  \dist (z,F^{-1}(0) ) \le \kappa\,\dist (0, F(z)) \quad
  \text{for all } z \in \cB_\epsilon (\zbar).
\]
It is \textit{strongly metrically subregular} if 
$\|z-\zbar\| \le \kappa \dist (0,F(z))$ for all $z \in \cB_\epsilon(\zbar)$.
\end{definition}
For general conic programs, metric subregularity of the KKT mappings
is studied, \eg, in \cite{HSZ18,DSZ17,CST19}.
In our setting, the two KKT mappings \eqref{eq:kkt-ext} and \eqref{eq:kkt}
are (strongly) metrically subregular at the same KKT points, because
residuals of one yield residuals of the other with comparable norms.  Indeed, for all
$(X,y,S)$,
\begin{equation} \label{eq:kkt-transfer}
\begin{aligned}
  (r_X,r_y)\in F_\KKT(X,y) &\;\Longrightarrow\;
  (r_y,0,r_X)\in\Ftilde_\KKT(X,y,C-\cA^\ast y),\\
  (r_\mathrm{p},r_\mathrm{d},r_\mathrm{c})\in\Ftilde_\KKT(X,y,S)
  &\;\Longrightarrow\;
  (r_\mathrm{c}-r_\mathrm{d},\,r_\mathrm{p})\in F_\KKT(X,y),
\end{aligned}
\end{equation}
where $r_\mathrm{p}$, $r_\mathrm{d}$, and $r_\mathrm{c}$ are primal
feasibility, dual feasibility, and complementarity residuals, respectively.
Moreover, $(X,y,S)\mapsto(X,y)$ maps $\Omega_\star$ bijectively onto
$\cZ_\star$, with inverse $(X,y)\mapsto(X,y,C-\cA^\ast y)$.

\section{Metric subregularity and uniform local linear convergence}
\label{sec:msr}

This section shows that metric subregularity of the KKT mapping
\eqref{eq:kkt} is exactly the property of the SDP that decides whether
PDHG converges linearly, with uniform constants, near a KKT point.
\Cref{sec:msr-char} proves this characterization for
$\tau\sigma\|\cA\|^2 < 1$, \cref{sec:boundary} extends it to the case
$\tau\sigma\|\cA\|^2 = 1$, and \cref{sec:pdnd}
derives linear convergence under primal--dual nondegeneracy from known
regularity results.  Throughout, $\cZ_\star=F_\KKT^{-1}(0)$, $R$ is the PDHG
map~\eqref{eq:pdhg-map}, and $d_P$ is the $P$-distance to $\cZ_\star$:
\[
  d_P(z)=\dist_P(z,\cZ_\star).
\]

\subsection{The characterization} \label{sec:msr-char}

We begin with the case~\eqref{eq:stepsize}, for which $P \succ 0$.

\begin{theorem}[Metric subregularity and uniform local linear convergence]
\label{thm:msr}
Suppose \cref{asp:prob} holds, $\tau\sigma\|\cA\|^2<1$, and
$z_\star\in\cZ_\star$.  The following are equivalent.
\begin{enumerate}[label=(\alph*)]
\item $F_\KKT$ is metrically subregular at $z_\star$ for $0$.

\item There exist $\epsilon>0$ and $\rho\in(0,1)$ such that
$d_P(Rz)\le \rho\,d_P(z)$ for all $z\in\cB^P_\epsilon(z_\star)$.

\item \textup{(Uniform local linear convergence.)}
There exist $\epsilon'>0$, $c \ge 1$, and $\rho' \in (0,1)$ such that
every PDHG sequence with $z_0 \in \cB^P_{\epsilon'}(z_\star)$ satisfies
$d_P(z_k) \le c \, (\rho')^k \, d_P(z_0)$ for all $k \in \natint$.
\end{enumerate}
Moreover, when these statements hold, with $\epsilon$ and $\rho$ as in~(b),
every PDHG sequence with
$z_{k_0} \in \cB^P_\epsilon(z_\star)$ for some $k_0 \in \natint$ converges
to a point $z_\infty \in \cZ_\star$, and
\begin{equation} \label{eq:Qlin-limit}
  \|z_{k+1}-z_\infty\|_P \le \sqrt{1-(1-\rho)^2/4} \,
  \|z_k-z_\infty\|_P \quad \text{for all } k\ge k_0.
\end{equation}
\end{theorem}
\begin{proof}
Since $P \succ 0$, metric subregularity in (a) is equivalent to
the existence of $\kappa, \epsilon>0$ with
$d_P(z) \le \kappa \dist_{P^{-1}} (0,F_\KKT(z))$ for all
$z \in \cB^P_\epsilon(z_\star)$.
Moreover, since $Rz_\star=z_\star$ and $R$ is nonexpansive in $\|\cdot\|_P$,
every $P$-ball centered at $z_\star$ is invariant under~$R$.

(a)$\Rightarrow$(b).
Let $z \in \cB^P_\epsilon(z_\star)$. Then $Rz \in \cB^P_\epsilon (z_\star)$.
By \eqref{eq:ppm}, $P(z-Rz)\in F_\KKT(Rz)$, and thus
$d_P(Rz) \le \kappa \|P(z-Rz)\|_{P^{-1}} = \kappa\|z-Rz\|_P$.
By \eqref{eq:fejer-dist},
\[
  d_P(Rz)^2\le d_P(z)^2-\|z-Rz\|_P^2\le d_P(z)^2-\kappa^{-2}d_P(Rz)^2,
\]
which is (b) with $\rho=\kappa/\sqrt{1+\kappa^2}$.

(b)$\Rightarrow$(c).  The ball $\cB^P_\epsilon(z_\star)$ is invariant
under $R$, so (b) can be iterated; (c) holds with $\epsilon'=\epsilon$,
$c=1$, and $\rho'=\rho$.

(c)$\Rightarrow$(b).
Choose $\kbar \in \natint$ with $c \cdot (\rho')^{\kbar} \le 1/2$.
Since $R$ is nonexpansive, the step lengths $\|R^{j+1}z-R^jz\|_P$ are
nonincreasing in~$j$.  Hence, for $z \in \cB^P_{\epsilon'}(z_\star)$,
\[
  d_P(z) \le \|z-R^{\kbar}z\|_P+d_P(R^{\kbar}z)
  \le \kbar\|z-Rz\|_P + \tfrac12 d_P(z),
\]
where the last step uses the monotonicity of the step lengths and (c).
Hence $d_P(z) \le 2\kbar\|z-Rz\|_P$, and combining this
with~\eqref{eq:fejer-dist} yields~(b)
with $\epsilon = \epsilon'$ and $\rho = (1-1/(4\kbar^2))^{1/2}$.

(b)$\Rightarrow$(a).
Let $u\in\cB^P_{\epsilon/2}(z_\star)$.  We show that
\[
  d_P(u)\le\max\{\rho/(1-\rho),1\}\,\dist_{P^{-1}}(0,F_\KKT(u)),
\]
which is (a), by considering two cases.
\begin{itemize}
\item \textit{Case 1: $\dist_{P^{-1}}(0,F_\KKT(u))\le\epsilon/2$.}
Since $F_\KKT(u)$ is closed and convex, there exists $v \in F_\KKT(u)$ with
$\|v\|_{P^{-1}}=\dist_{P^{-1}}(0,F_\KKT(u))$.
The point $w = u + P^{-1}v$ lies in $\cB^P_\epsilon(z_\star)$
because $\|P^{-1}v\|_P=\|v\|_{P^{-1}}\le\epsilon/2$, and satisfies
$P(w-u) = v \in F_\KKT(u)$.  Thus,~$u$ solves~\eqref{eq:ppm-r}
with $z \gets w$ and $r \gets 0$, and $Rw=u$ by uniqueness.  Hence,
$d_P(u)=d_P(Rw)\le \rho\,d_P(w)\le \rho \, (d_P(u)+\|v\|_{P^{-1}})$,
and so $d_P(u) \le \tfrac{\rho}{1-\rho} \|v\|_{P^{-1}}$.

\item \textit{Case 2: $\dist_{P^{-1}}(0,F_\KKT(u))>\epsilon/2$.}
Then
\[
  d_P(u)\le\|u-z_\star\|_P\le\epsilon/2<\dist_{P^{-1}}(0,F_\KKT(u)).
\]
\end{itemize}
Finally, let $\epsilon$ and $\rho$ be as in (b),
let $z_{k_0} \in \cB^P_\epsilon(z_\star)$,
and let $z_\infty \in \cZ_\star$ be the limit of $\{z_k\}$.
Since $\cB^P_\epsilon(z_\star)$ is invariant under $R$,
$z_k\in\cB^P_\epsilon(z_\star)$ for all $k\ge k_0$.
For $j\ge k\ge k_0$, \eqref{eq:fejer} with
$\zbar \gets \proj^P_{\cZ_\star}(z_k)$ gives $\|z_j-\zbar\|_P\le d_P(z_k)$;
letting $j\to\infty$ yields $\|z_\infty-\zbar\|_P\le d_P(z_k)$, and hence
\[
  \|z_k-z_\infty\|_P \leq \|z_k-\zbar\|_P + \|\zbar-z_\infty\|_P
  \le 2d_P(z_k).
\]
Combining it with the triangle inequality and (b) yields
\[
  \|z_k-z_{k+1}\|_P\ge d_P(z_k)-d_P(z_{k+1})\ge(1-\rho)\,d_P(z_k)
  \ge\tfrac{1-\rho}{2}\|z_k-z_\infty\|_P.
\]
This and \eqref{eq:fejer} with $z \gets z_k$ and $\zbar \gets z_\infty$
give~\eqref{eq:Qlin-limit}.
\end{proof}

Each implication in \cref{thm:msr} has a counterpart for proximal point
methods or for general fixed-point iterations.
The implication (a)$\Rightarrow$(b) is the classical sufficiency of error
bounds for the proximal point method \cite{Rockafellar76b,Luq84,Lev09},
and a variable-metric version is given in \cite[Theorem~2.2]{Roc23}.
For the converse, metric subregularity of a maximal monotone mapping~$F$ is
equivalent to that of $(\Id+F^{-1})^{-1}=\Id-(\Id+F)^{-1}$.
(Drusvyatskiy and Lewis \cite{DL18} prove this for subdifferentials
of convex functions, and the same argument applies to maximal monotone~$F$.)
With $F=P^{-1}F_\KKT$, it turns (a) into a local error bound
$d_P(z) \le \kappa' \|z-Rz\|_P$ on the fixed-point residual of~$R$.
For general fixed-point iterations, such error bounds are necessary
for uniform linear convergence \cite{LTT20} and, for averaged operators,
also sufficient, with Q-linear convergence to the limit \cite{vTPVZ25}.
Combined with \cite{DL18}, these results correspond to
(c)$\Rightarrow$(a) and to (a)$\Rightarrow$(c) with~\eqref{eq:Qlin-limit}.
The key point of \cref{thm:msr} is that condition (a) is expressed through
the KKT mapping of the SDP rather than through the iteration operator.
Consequently, the existence of a uniform local linear regime does not
depend on the stepsizes.

Two further remarks concern the scope of \cref{thm:msr}.
First, the condition is pointwise: it concerns the KKT point $z_\star$,
so whether the theorem applies depends on the KKT point to which
the algorithm converges.
Second, the theorem concerns uniform linear convergence,
with constants common to all starting points near $z_\star$ as in (c).
Metric subregularity is necessary for this uniform behavior,
but a particular PDHG sequence may still converge linearly to a KKT point
at which metric subregularity fails.

\subsection{The degenerate case} \label{sec:boundary}

When $\tau\sigma\|\cA\|^2=1$, $P$ is positive semidefinite and singular,
and PDHG is a degenerate preconditioned proximal point method \cite{BCLN22}.
The function $\|\cdot\|_P$ is only a seminorm, and
$d_P(z) = \inf_{\zbar \in \cZ_\star}\|z-\zbar\|_P$ cannot bound
the Euclidean distance to $\cZ_\star$.
The PDHG step is nevertheless determined by~$Pz$ and Lipschitz continuous
in $Pz$ by \eqref{eq:step-lip}, which is enough to extend \cref{thm:msr}.

\begin{proposition} \label{prop:boundary}
Suppose \cref{asp:prob} holds, $z_\star\in\cZ_\star$, and
$\tau\sigma\|\cA\|^2=1$.  Then statements (a)--(c) of \cref{thm:msr} are
equivalent when the balls $\cB^P_\epsilon(z_\star)$ and
$\cB^P_{\epsilon'}(z_\star)$ are replaced by the Euclidean balls
$\cB_\epsilon(z_\star)$ and $\cB_{\epsilon'}(z_\star)$.
Moreover, when these statements hold, every PDHG sequence $\{z_k\}$
converging to $z_\star$ satisfies,
for some $\rho \in (0,1)$, $c>0$, and $k_0\in\natint$,
\[
  d_P(z_{k+1})\le\rho\,d_P(z_k) \quad\text{and}\quad
  \|z_k-z_\star\|\le c\,\rho^k \qquad\text{for all } k\ge k_0.
\]
\end{proposition}
\begin{proof}
Let $L$ be the constant in \eqref{eq:step-lip}, and set
$\ell=L\|P\|^{1/2}$.  From \eqref{eq:step-lip},
\begin{equation} \label{eq:R-lip}
  \|Rz-Rz'\|\le\ell\,\|z-z'\|_P\le L\|P\|\,\|z-z'\| \qquad
  \text{for all } z,z'.
\end{equation}

(a)$\Rightarrow$(b).
Let $\kappa,\epsilon_0$ be as in \cref{def:msr},
and let $\epsilon=\epsilon_0/(L\|P\|)$.
For $z \in \cB_\epsilon(z_\star)$, \eqref{eq:R-lip} with $z' \gets z_\star$
gives $Rz \in \cB_{\epsilon_0}(z_\star)$.
Since $P(z-Rz) \in F_\KKT(Rz)$ by~\eqref{eq:ppm},
\[
  d_P(Rz) \le \|P\|^{1/2} \dist(Rz,\cZ_\star)
  \le \kappa \|P\|^{1/2} \|P(z-Rz)\| \le \kappa \|P\| \, \|z-Rz\|_P .
\]
Together with \eqref{eq:fejer-dist}, this gives (b) with
$\rho=\kappa\|P\|/(1+\kappa^2\|P\|^2)^{1/2}$.

(b)$\Rightarrow$(c).  Choose $\epsilon' \le \epsilon$ with
$\epsilon' L \|P\| (1+1/(1-\rho)) \le \epsilon$,
and let $z_0 \in \cB_{\epsilon'}(z_\star)$.  We show by induction that
$z_k \in \cB_\epsilon(z_\star)$ for all $k$.  If $z_0,\ldots,z_k$ lie in the
ball, then $d_P(z_j)\le\rho^jd_P(z_0)\le\rho^j\|P\|^{1/2}\epsilon'$ for
$j\le k$.  By \eqref{eq:R-lip} with $z'\gets z_\star$,
$\|z_1-z_\star\| \le L\|P\|\epsilon'$.
By \eqref{eq:R-lip} and~\eqref{eq:fejer-dist},
\[
  \|z_{j+2}-z_{j+1}\| \le \ell \, \|z_{j+1}-z_j\|_P \le \ell \, d_P(z_j)
  \le \rho^j L \|P\| \epsilon'.
\]
Hence, by the triangle inequality,
\[
  \|z_{k+1}-z_\star\|\le\|z_1-z_\star\|+\sum_{j=0}^{k-1}\|z_{j+2}-z_{j+1}\|
  \le L\|P\|\epsilon'+\frac{L\|P\|\epsilon'}{1-\rho}\le\epsilon .
\]
Thus (b) can be iterated, and (c) holds with $c=1$ and $\rho'=\rho$.

(c)$\Rightarrow$(b).
The proof of \cref{thm:msr} applies verbatim with $\cB_{\epsilon'}(z_\star)$
in place of $\cB^P_{\epsilon'}(z_\star)$,
since it only uses \eqref{eq:fejer-dist} and firm nonexpansiveness
of~$R$ in the $P$-seminorm.

(b)$\Rightarrow$(a).
Let $\Pi_P=PP^\dagger$ be the orthogonal projector onto $\Range(P)$,
set $\delta = \epsilon/(2\|P^\dagger\|)$,
and let $u \in \cB_{\epsilon/2}(z_\star)$.
We establish statement (a) in two cases.
\begin{itemize}
\item \textit{Case 1: $\dist(0,F_\KKT(u))\le\delta$.}
Let $v$ be the element of minimal norm in $F_\KKT(u)$, and let
$w=u+P^\dagger v$, so that
$\|w-z_\star\|\le\epsilon/2+\|P^\dagger\|\,\|v\|\le\epsilon$.
Since $u$ solves \eqref{eq:ppm-r} for $(z,r) \gets (u,v)$, $Rw$ solves it
for $(z,r) \gets (w,0)$, and $P(w-u)=\Pi_Pv$, \eqref{eq:step-lip} gives
$\|u-Rw\|\le L\|(\Id-\Pi_P)v\|\le L\|v\|$.  Hence
\[
  \|w-Rw\|_P\le\|P\|^{1/2}\big(\|w-u\|+\|u-Rw\|\big)
  \le\|P\|^{1/2}\big(\|P^\dagger\|+L\big)\|v\|.
\]
By (b), $d_P(w) \le \|w-Rw\|_P + d_P(Rw) \le \|w-Rw\|_P + \rho\,d_P(w)$,
which implies $d_P(w) \le \|w-Rw\|_P/(1-\rho)$.
Finally, \eqref{eq:R-lip} gives $\dist(Rw,\cZ_\star)\le\ell\,d_P(w)$
(since $R\zbar=\zbar$ for $\zbar\in\cZ_\star$), and therefore
\[
  \dist(u,\cZ_\star)\le\|u-Rw\|+\ell\,d_P(w)
  \le\Big(L+\frac{L\|P\|(\|P^\dagger\|+L)}{1-\rho}\Big)\|v\|.
\]

\item \textit{Case 2: $\dist(0,F_\KKT(u))>\delta$.}  Then
\[
  \dist(u,\cZ_\star)\le\|u-z_\star\|\le\epsilon/2=\|P^\dagger\|\,\delta
  <\|P^\dagger\|\dist(0,F_\KKT(u)).
\]
\end{itemize}
Hence (a) holds, with $\kappa$ the larger of the two constants above.

Finally, let $\{z_k\}$ converge to $z_\star$, and choose $k_0$ with
$z_k\in\cB_\epsilon(z_\star)$ for $k\ge k_0$.
Then $d_P(z_{k+1})\le\rho\,d_P(z_k)$ for $k\ge k_0$ by (b).
Hence, by \eqref{eq:R-lip} and \eqref{eq:fejer-dist},
\[
  \|z_{j+2}-z_{j+1}\|\le\ell\,d_P(z_j)\le\ell\rho^{j-k_0}d_P(z_{k_0})
  \quad \text{for } j \ge k_0.
\]
Since $z_j\to z_\star$, summing these bounds gives
$\|z_k-z_\star\|\le\sum_{j\ge k}\|z_{j+1}-z_j\|\le c\rho^k$ for
$k\ge k_0+1$, with $c = \ell\rho^{-k_0-1}d_P(z_{k_0})/(1-\rho)$;
replacing $k_0$ by $k_0+1$ gives the claim.
\end{proof}

The proof differs from that of \cref{thm:msr} in three places,
all due to the singularity of $P$.
First, the $P$-balls contain the affine subspace $z_\star+\Nullspace(P)$,
so they are unbounded, and conditions on them are not local.
Euclidean balls are used instead, and \eqref{eq:R-lip} converts $P$-seminorm
estimates into Euclidean ones.  Second, Euclidean balls need not be
invariant under $R$, which is nonexpansive only in the $P$-seminorm, so in
(b)$\Rightarrow$(c) the iterates are kept in the ball by summing the
Euclidean step lengths.  Third, $P^{-1}$ does not exist, so the point
$u+P^{-1}v$ in (b)$\Rightarrow$(a) is replaced by $u+P^\dagger v$, and the
component $(\Id-\Pi_P)v$ of the residual, which the iteration does not see,
is controlled by \eqref{eq:step-lip} with $r\neq r'$.  We do not state
\eqref{eq:Qlin-limit} in the degenerate case:
in the $P$-seminorm it does not control the components in $\Nullspace(P)$,
and in the Euclidean norm we only obtain R-linear convergence.

\subsection{Primal--dual nondegeneracy}
\label{sec:pdnd}

Primal--dual nondegeneracy gives metric subregularity through known
variational results. Then the classical proximal-point mechanism
\cite{Rockafellar76b,Lev09}, in the form of \cref{thm:msr,prop:boundary},
gives linear convergence.  The SDP-specific ingredients are
\cite[Theorems~10, 11]{AHO97} and the equivalence
between nondegeneracy and strong regularity \cite{CS08}.

\begin{corollary}[Local linear convergence under nondegeneracy]
\label{thm:pdnd}
Suppose \cref{asp:prob} holds, $\tau \sigma \|\cA\|^2 \le 1$,
and $(X_\star,y_\star,S_\star)\in\Omega_\star$ satisfies primal--dual
nondegeneracy \eqref{eq:pdnd}.  Then $\cZ_\star=\{z_\star\}$,
$F_\KKT$ is strongly metrically subregular at $z_\star$ for~$0$,
and every PDHG sequence converges to $z_\star$.
If $\tau\sigma\|\cA\|^2<1$, then
$\|z_{k+1}-z_\star\|_P \le \rho \|z_k-z_\star\|_P$ for some $\rho \in (0,1)$
and all sufficiently large $k$;
if $\tau\sigma\|\cA\|^2=1$, then $z_k\to z_\star$ R-linearly.
\end{corollary}
\begin{proof}
Uniqueness of the primal and dual solutions follows from \eqref{eq:pdnd}
\cite{AHO97}, so every PDHG sequence converges to $z_\star$
(\cref{sec:ppm}).  By \cite[Theorem~18]{CS08}, primal--dual nondegeneracy
is equivalent to strong regularity of $0\in\Ftilde_\KKT(X,y,S)$ at
$(X_\star,y_\star,S_\star)$, and strong regularity implies strong metric
subregularity \cite[Theorems~4.2 and~5.2]{DR04}.  By
\eqref{eq:kkt-transfer}, this property transfers to $F_\KKT$.  The rates
follow from \cref{thm:msr} if $\tau\sigma\|\cA\|^2<1$ and from
\cref{prop:boundary} if $\tau\sigma\|\cA\|^2=1$.
\end{proof}

For ADMM, linear convergence under metric subregularity, and hence under
nondegeneracy, was shown in \cite{HSZ18}; see also \cite[Theorem~6]{KJY26}.

\section{Local linear convergence under strict complementarity}
\label{sec:sc}

This section assumes that the KKT point $(X_\star,y_\star,S_\star)$
satisfies strict complementarity \eqref{eq:sc}.
This assumption is used in two ways.
First, $Z_\star = X_\star - \tau S_\star$ is nonsingular,
so the projection onto $\SV{n}{+}$ is smooth near $Z_\star$.
Second, the local KKT set has a fixed face structure.
These two facts reduce the convergence analysis
to a linearized fixed-point map on a flat local model.

For ease of presentation, we assume without loss of generality
that the limit points $X_\star$ and $S_\star$ are diagonal,
\ie, $Q_\star = I$ in \eqref{eq:eigen}.
Indeed, replacing $C$ and $A_i$ by $Q_\star^\tran CQ_\star$ and
$Q_\star^\tran A_iQ_\star$ maps the PDHG iterates $(X_k,y_k)$ to
$(Q_\star^\tran X_kQ_\star,y_k)$ and preserves all norms, distances,
and spectra used below, because $\proj_{\SV{n}{+}}$ commutes with
orthogonal similarity.

\subsection{Local KKT geometry under strict complementarity}

Strict complementarity identifies the two minimal faces of $\SV{n}{+}$ that
contain the nearby primal and dual slack solutions.  Consequently, the KKT
set is locally a relatively open subset of an affine set.
\begin{lemma} \label{lem:local-geometry}
Suppose \cref{asp:prob} holds and $z_\star$ satisfies strict
complementarity.  There is a neighborhood of $z_\star$ such that the KKT set
is represented by%
\begin{equation} \label{eq:M}
  \cM = \{(X,y) \mid \cA X=b, \, X \in \cU, \, C-\cA^\ast y \in \cV, \,
  X_{11} \succ 0, \, (C-\cA^\ast y)_{22} \succ 0\},
\end{equation}
where the linear subspaces $\cU$ and $\cV$ are defined by
\[
  \cU = \left\{\begin{bmatrix} U_{11} & 0 \\ 0 & 0 \end{bmatrix}
  \;\middle|\; U_{11} \in \symm^r \right\}, \qquad
  \cV = \left\{\begin{bmatrix} 0 & 0 \\ 0 & V_{22} \end{bmatrix}
  \;\middle|\; V_{22} \in \symm^{n-r} \right\}.
\]
Equivalently, for all $(X,y)$ sufficiently close to $(X_\star,y_\star)$,
\[
  (X,y,C-\cA^\ast y) \in \Omega_\star
  \qquad \Longleftrightarrow \qquad
  (X,y) \in \cM.
\] 
Moreover, $\cM$ is a relatively open subset of an affine space,
and its tangent space is
\begin{equation} \label{eq:tangent}
  \cT_\cM = \{(U,v) \mid U \in \cU, \ \cA U=0, \ \cA^\ast v \in \cV\}.
\end{equation}
\end{lemma}
\begin{proof}
Let $(X,y,S)$ be a KKT point near $(X_\star,y_\star,S_\star)$.
Since $X$ is primal optimal and $(y_\star,S_\star)$ is dual optimal,
it holds that $\inprod{C}{X} - b^\tran y_\star = \inprod{X}{S_\star} = 0$.
As $X \succeq 0$ and $S_\star \succeq 0$, this implies $X S_\star=0$.
Since $S_\star$ is positive definite on the lower-right block
in~\eqref{eq:eigen}, we obtain $X \in \cU$.
Similarly, primal optimality of $X_\star$ and dual optimality of $(y,S)$
give $\inprod{X_\star}{S}=0$.  Since $S \succeq 0$ and $X_\star$ is positive
definite on the upper-left block, we obtain $S \in \cV$. 
For all KKT points sufficiently close to $(X_\star,y_\star,S_\star)$,
the positive blocks $X_{11}$ and $S_{22}$ remain positive definite.
Hence $(X,y) \in \cM$.

Conversely, if $(X,y) \in \cM$ and $S=C-\cA^\ast y$, then $X \succeq 0$,
$S \succeq 0$, $\cA X=b$, $\cA^\ast y+S=C$, and $XS=0$
by the block structure.  Thus $(X,y,S) \in \Omega_\star$.
Finally, $\cM$ is a relatively open subset of the affine space
defined by the three linear constraints in \eqref{eq:M}, and
differentiating these constraints gives \eqref{eq:tangent}.
\end{proof}

The flat set $\cM$ should not be confused with the curved manifold of
rank-$r$ positive semidefinite matrices, which is the object of rank
identification; see \cref{sec:rank-id}.

\subsection{Local linearization of PDHG}

By \cref{sec:ppm}, PDHG is the fixed-point iteration $z_{k+1}=Rz_k$, and
under \eqref{eq:stepsize}, $P\succ0$ and $R$ is firmly nonexpansive in the
$P$-norm \eqref{eq:fne}.  We now differentiate $R$ near $\cM$, using the
derivative of $\proj_{\SV{n}{+}}$ \cite[Theorem~4.6]{SS02},
\cite[Lemma~1]{KJY26}.
Let $Z \in \symm^n$ be nonsingular with $r$ positive eigenvalues and
eigen-decomposition
\[
  Z = Q \diag (\lambda_1,\ldots,\lambda_n) Q^\tran, \quad \text{where }
  \lambda_1 \ge \cdots \ge \lambda_r > 0 > \lambda_{r+1} \ge \cdots \ge
  \lambda_n
\]
and $Q \in \R[n \times n]$ is an orthogonal matrix.
Then $\proj_{\SV{n}{+}}$ is Fr\'{e}chet differentiable at $Z$,
and its derivative is
\begin{equation} \label{eq:dPi}
  \D \proj_{\SV{n}{+}} (Z) [H]
  = Q (\Gamma \circ (Q^\tran H Q)) Q^\tran, \qquad
  \Gamma = \begin{bmatrix}
    \ones_r \ones_r^\tran & \Theta^\tran \\ \Theta & 0
  \end{bmatrix},
\end{equation}
where $\Theta \in \R[(n-r) \times r]$ has entries
$\Theta_{ij} = \lambda_j/(\lambda_j - \lambda_{i+r}) \in (0,1)$.

Let $\zbar = (\Xbar,\ybar) \in \cM$, and set $\Sbar = C - \cA^\ast \ybar$,
$\Zbar = \Xbar - \tau \Sbar$, and
$J_{\zbar} = \D\proj_{\SV{n}{+}} (\Zbar)$.
By \eqref{eq:M}, $\Zbar$ is block diagonal with diagonal blocks
$\Xbar_{11} \succ 0$ and $-\tau \Sbar_{22}\prec0$, so $\Zbar$ is nonsingular
with $r$ positive eigenvalues.  Hence $R$ is continuously differentiable
near~$\zbar$, and by \eqref{eq:step} its derivative
$G_{\zbar} = \D R(\zbar)$ is
\begin{equation} \label{eq:G}
  G_{\zbar} \begin{bmatrix} U \\ v \end{bmatrix} =
  \begin{bmatrix}
    J_{\zbar}(U+\tau\cA^\ast v)\\
    v+\sigma\cA U - 2\sigma\cA J_{\zbar} (U+\tau\cA^\ast v)
  \end{bmatrix}.
\end{equation}
We write $Z_\star = X_\star - \tau S_\star$, $J_\star = J_{z_\star}$, and
$G_\star = G_{z_\star}$.

\subsection{Fixed directions and normal contraction}

In this subsection, we use the differential description of the PDHG map~$R$
to separate the tangent and normal directions to the local KKT set~$\cM$.
We show that the tangent space $\cT_\cM$ is exactly the set of fixed points
of $G_\star$, and that all directions normal to~$\cM$
are contracted uniformly.

The first lemma identifies the fixed-point set of the local linearization.
\begin{lemma} \label{lem:fixG}
Suppose \cref{asp:prob} holds and $z_\star$ satisfies strict
complementarity.  For every $\zbar \in \cM$, $\fix(G_{\zbar}) = \cT_\cM$.
In particular, $\fix(G_\star) = \cT_\cM$.
\end{lemma}
\begin{proof}
Let $\zbar \in \cM$, and let $(U,v)$ be a fixed point of $G_{\zbar}$;
\ie, $G_{\zbar}(U,v) = (U,v)$.  Set $H = U + \tau \cA^\ast v$.
Then the fixed-point equation gives
\[
  U = J_{\zbar} H, \qquad \cA (U - 2J_{\zbar} H) = 0.
\] 
Substituting $U = J_{\zbar} H$ into the second equation gives $\cA U=0$.
Moreover, $\tau \cA^\ast v = H - U = (I-J_{\zbar}) H$.
Taking the inner product with $U = J_{\zbar} H$ gives
\begin{equation} \label{eq:fixG-prf}
  0 = \tau \inprod{\cA U}{v} = \inprod{J_{\zbar} H}{(I-J_{\zbar}) H}.
\end{equation} 
Since $\Zbar$ is block diagonal, the matrix $Q$ in \eqref{eq:dPi}
can be chosen block diagonal, conformally with $\cU$ and $\cV$.  Replacing
$H$ by $Q^\tran HQ$, which preserves the inner product and maps $\cU$ and
$\cV$ onto themselves, we may assume that $Q=I$.
Let $H_{11}$, $H_{21}$, and $H_{22}$ denote the blocks of $H$,
partitioned conformally with $\cU$ and $\cV$.  Then,
\begin{align*}
  J_{\zbar} H &= \begin{bmatrix} 
	H_{11} & (\Theta \circ H_{21})^\tran \\ \Theta \circ H_{21} & 0
  \end{bmatrix}, \\
  (I - J_{\zbar}) H &= \begin{bmatrix} 
	0 & ((\ones\ones^\tran-\Theta) \circ H_{21})^\tran \\
	(\ones\ones^\tran-\Theta) \circ H_{21} & H_{22}
  \end{bmatrix},
\end{align*}
where the matrix $\Theta \in \R[(n-r) \times r]$ from
\eqref{eq:dPi} satisfies $\Theta_{ij} \in (0,1)$ for all $i,j$.
Hence $\inprod{J_{\zbar}H}{(I-J_{\zbar})H}
= 2\sum_{i,j}\Theta_{ij}(1-\Theta_{ij})(H_{21})_{ij}^2$,
and the zero inner product in \eqref{eq:fixG-prf} forces $H_{21}=0$.
Consequently,
\[
  U = J_{\zbar} H \in \cU, \qquad
  \cA^\ast v = \tau^{-1} (I-J_{\zbar}) H \in \cV.
\]
Combining this with $\cA U = 0$ yields $(U,v) \in \cT_\cM$.

Conversely, if $(U,v) \in \cT_\cM$, then $U \in \cU$, $\cA U=0$, 
and $\cA^\ast v \in \cV$.  Hence $J_{\zbar} (U + \tau\cA^\ast v)=U$,
and the second component of $G_{\zbar}$ in \eqref{eq:G} becomes
$v+\sigma\cA U-2\sigma\cA U=v$ since $\cA U=0$.
Hence $G_{\zbar} (U,v) = (U,v)$.
\end{proof}

The second lemma characterizes the contraction properties
of vectors normal to~$\cT_\cM$.
Let $\proj^P$ denote the $P$-orthogonal projector onto $\cT_\cM$:
\[
  \proj^P (z) = \argmin_{z^\prime \in \cT_\cM} \tfrac12 \|z-z^\prime\|_P^2.
\] 

\begin{lemma} \label{lem:normal-contraction}
Suppose \cref{asp:prob} holds, the stepsizes satisfy \eqref{eq:stepsize},
and $z_\star$ satisfies strict complementarity.
There exists $\epsilon > 0$ such that
$\cC_\epsilon = (\aff\cM) \cap \cB^P_\epsilon (z_\star)$ is a subset
of~$\cM$.  Moreover, there exist $\delta > 0$, $\alpha \in (0,1)$, and a
nondecreasing function $\omega$ with $\omega(t)\to0$ as $t\downarrow0$
such that, for every $\zbar \in \cC_\epsilon$ and every
$u \perp_P \cT_\cM$ with $\|u\|_P \le \delta$, the remainder
$r_{\zbar}(u) = R(\zbar+u) - \zbar - G_{\zbar}u$ satisfies
\begin{equation} \label{eq:normal-G}
  \|(I-\proj^P) G_{\zbar} u\|_P \le \alpha \|u\|_P, \qquad
  \|r_{\zbar} (u)\|_P \le \omega(\|u\|_P)\,\|u\|_P
  \le \tfrac{1-\alpha}{2} \|u\|_P.
\end{equation}
Consequently, there exists a constant $\rho \in (0,1)$ such that
\begin{equation}\label{eq:normal-contraction}
  \dist_P(Rz,\cM) \le \rho \dist_P(z,\cM)
  \quad \text{for all $z$ sufficiently close to $z_\star$.}
\end{equation}
\end{lemma}
\begin{proof}
Since $\cM$ is relatively open in its affine hull and $z_\star \in \cM$,
there is $\epsilon > 0$ with $\cC_\epsilon \subseteq \cM$.
The set $\cC_\epsilon$ is compact, $R$ is continuously differentiable on a
neighborhood of $\cC_\epsilon$, and $R\zbar=\zbar$ for all
$\zbar\in\cC_\epsilon$.

Fix $\zbar \in \cC_\epsilon$.  Applying \eqref{eq:fne} to $\zbar+th$ and
$\zbar$, dividing by~$t^2$, and letting $t \downarrow 0$ gives%
\begin{equation} \label{eq:G-fne}
  \|G_{\zbar}h\|_P^2 \le \inprod{G_{\zbar} h}{h}_P
  \le \|G_{\zbar}h\|_P\|h\|_P \quad \text{for all } h,
\end{equation}
so $\|G_{\zbar} h\|_P \le \|h\|_P$.  If equality holds for some $h\ne0$,
then the Cauchy--Schwarz inequality in \eqref{eq:G-fne} is tight, so
$G_{\zbar}h = h$ and $h \in \cT_\cM$ by \cref{lem:fixG}.

We now prove that the contraction is uniform in $\zbar$.
Consider the compact set
\[
  \cD_\epsilon = \{(\zbar,u) \mid \zbar \in \cC_\epsilon, \,
  u \perp_P \cT_\cM, \, \|u\|_P=1\}.
\]
Then the function $(\zbar,u) \mapsto \|(I-\proj^P) G_{\zbar} u\|_P$
is continuous on $\cD_\epsilon$.  If its maximum were equal to one,
then for some $(\zbar,u) \in \cD_\epsilon$,
\[
  1 = \|(I-\proj^P) G_{\zbar} u\|_P \le \|G_{\zbar} u\|_P \le \|u\|_P = 1.
\]
Thus $\|G_{\zbar} u\|_P = \|u\|_P$.  By the equality case above,
$u \in \cT_\cM$, contradicting $u \perp_P \cT_\cM$ and $\|u\|_P=1$.
Hence the maximum is strictly less than one;
let $\alpha\in(0,1)$ be at least this maximum.
By homogeneity, this gives the first bound in~\eqref{eq:normal-G}.

It remains to make the linear approximation uniform.
Since $R$ is continuously differentiable on a neighborhood of the compact
set $\cC_\epsilon$, its derivative is uniformly continuous there.
Hence, by the mean value theorem, there exist $t_0>0$ and a
nondecreasing function $\omega$ with $\omega(t)\to0$ as $t\downarrow0$
such that, for all $\zbar \in \cC_\epsilon$ and all $u$ with
$\|u\|_P \le t_0$,
\[
  \|r_{\zbar}(u)\|_P \le \omega(\|u\|_P)\,\|u\|_P.
\]
Choosing $\delta\in(0,t_0]$ with $\omega(\delta)\le(1-\alpha)/2$ proves the
second bound in~\eqref{eq:normal-G}.

We now prove the distance contraction.  Choose a neighborhood $\cO$ of
$z_\star$ such that, for every $z\in \cO$, the $P$-orthogonal projection
$\zbar := \proj^P_{\aff\cM}(z)$ belongs to $\cC_\epsilon$ and
the normal component $u:=z-\zbar$ satisfies $\|u\|_P \le \delta$.
Since $\aff\cM$ is affine, $u \perp_P \cT_\cM$.
Since the nearest point $\zbar$ of $\aff\cM$ lies in
$\cM\subseteq\aff\cM$, we have
$\dist_P(z,\cM) = \dist_P(z,\aff\cM) = \|u\|_P$.

By the definition of $r_{\zbar}$, $Rz = \zbar + G_{\zbar}u + r_{\zbar}(u)$.
The vector $\proj^P (G_{\zbar} u+r_{\zbar} (u))$ belongs to $\cT_\cM$.
Hence, $\zhat := \zbar + \proj^P (G_{\zbar} u + r_{\zbar}(u)) \in \aff\cM$.
By shrinking the neighborhood~$\cO$ if necessary, we may ensure that
$\zhat \in \cM$, since $\cM$ is relatively open in $\aff \cM$ and
$\zbar \in \cC_\epsilon \subseteq \cM$.  Therefore,
\begin{align*}
  \dist_P(Rz,\cM) &\le \|Rz - \zhat\|_P \\
  &= \|(I-\proj^P) (G_{\zbar} u + r_{\zbar}(u))\|_P \\
  &\le \|(I-\proj^P)G_{\zbar}u\|_P + \|(I-\proj^P) r_{\zbar}(u)\|_P \\
  &\le \alpha\|u\|_P + \tfrac{1-\alpha}{2}\|u\|_P \\
  &= \tfrac{1+\alpha}{2}\dist_P(z,\cM).
\end{align*}
Thus \eqref{eq:normal-contraction} holds on $\cO$
with $\rho=(1+\alpha)/2<1$.
\end{proof}

\subsection{Local linear convergence}

The iterates eventually enter the neighborhood of
\cref{lem:normal-contraction}.  There, the distance to $\cM$ equals the
distance to $\cZ_\star$, so the normal contraction is statement~(b) of
\cref{thm:msr}, which yields linear convergence.

\begin{theorem}[Local linear convergence under strict complementarity]
\label{thm:sc}
Suppose \cref{asp:prob,asp:sc} hold and the stepsizes satisfy
\eqref{eq:stepsize}.  Let $z_k=(X_k,y_k)$ be the PDHG sequence, and let
$z_\star=(X_\star,y_\star)$ be its limit.  Then there exist
$\rho \in (0,1)$ and $k_0 \in \natint$ such that, for all $k \ge k_0$,
\begin{equation} \label{eq:thm-dist-Qlin}
  \dist_P(z_{k+1},\cM) \le \rho \dist_P(z_k,\cM).
\end{equation}
Moreover, $z_k$ converges Q-linearly to $z_\star$ in the $P$-norm: for all
$k \ge k_0$,
\begin{equation} \label{eq:thm-Qlin-limit}
  \|z_{k+1}-z_\star\|_P \le \tilde\rho \|z_k-z_\star\|_P, \quad
  \text{with } \tilde\rho = \sqrt{1-(1-\rho)^2/4}.
\end{equation}
\end{theorem}
\begin{proof}
Let $\rho\in(0,1)$ be the constant in \eqref{eq:normal-contraction}.
By \cref{lem:local-geometry,lem:normal-contraction}, there exists
$\epsilon>0$ such that $\cZ_\star\cap\cB^P_\epsilon(z_\star)\subseteq\cM$
and
\begin{equation} \label{eq:thm-sc-prf}
  \dist_P(Rz,\cM)\le\rho\,\dist_P(z,\cM)
  \qquad\text{for all } z\in\cB^P_\epsilon(z_\star).
\end{equation}
Since $z_k\to z_\star$, there exists $k_0\in\natint$ such that
$z_k\in\cB^P_\epsilon(z_\star)$ for all $k\ge k_0$.  Then
\eqref{eq:thm-sc-prf} with $z \gets z_k$ gives \eqref{eq:thm-dist-Qlin}.

To prove \eqref{eq:thm-Qlin-limit}, we show that $\dist_P(\cdot,\cM)$
coincides with $d_P$ on $\cB^P_\epsilon(z_\star)$.  Let
$z\in\cB^P_\epsilon(z_\star)$, and let $\zbar$ be the $P$-projection of $z$
onto the closed convex set $\cZ_\star$.  Since the projection is
nonexpansive and $z_\star\in\cZ_\star$,
\[
  \|\zbar-z_\star\|_P\le\|z-z_\star\|_P\le\epsilon,
\]
so $\zbar \in \cZ_\star \cap \cB^P_\epsilon(z_\star) \subseteq \cM$.
Since every point of $\cM$ is a KKT point,
\[
  d_P(z)\le\dist_P(z,\cM)\le\|z-\zbar\|_P=d_P(z).
\]
Hence $d_P(z) = \dist_P(z,\cM)$.
Since $\cB^P_\epsilon(z_\star)$ is invariant under $R$,
the same holds at~$Rz$, and \eqref{eq:thm-sc-prf} becomes
$d_P(Rz) \le \rho \, d_P(z)$ for all $z \in \cB^P_\epsilon(z_\star)$.
This is statement~(b) of \cref{thm:msr}.
Since $z_{k_0} \in \cB^P_\epsilon(z_\star)$ and $z_k \to z_\star$,
the last part of \cref{thm:msr} with $z_\infty \gets z_\star$
gives \eqref{eq:thm-Qlin-limit}.
\end{proof}

With $S_k=C-\cA^\ast y_k$, the equivalence of norms and
$S_k-S_\star=-\cA^\ast(y_k-y_\star)$ show that
\eqref{eq:thm-Qlin-limit} implies R-linear convergence of
$(X_k,y_k,S_k)$ to $(X_\star,y_\star,S_\star)$ in the standard product
norm; Q-linear convergence is not preserved by a change of norm in general.

Moreover, by \cref{prop:boundary} and \cref{cor:sc-msr} below,
linear convergence persists in the degenerate case: if
$\tau\sigma\|\cA\|^2=1$ and PDHG converges to a strictly complementary KKT
point, then $d_P(z_k)$ converges Q-linearly and $z_k$ converges R-linearly.

\paragraph{Comparison of the proof strategies.} \label{sec:admm}
Local linear convergence of ADMM under strict complementarity was
established in \cite{KJY26} by a different argument.
That analysis linearizes the Douglas--Rachford operator at the limit
$Z_\star$ and works with $H_k=Z_k-Z_\star$.
The fixed space of that linearization \cite[Proposition~1]{KJY26}
corresponds to~$\cT_\cM$ under $(U,v) \mapsto U + \tau\cA^\ast v$.
When the KKT set is not a singleton, $H_k$ has components in this space,
which are not contracted, so the remainder $O(\|H_k\|^2)$ is not dominated
by the contracting component; hence \cite{KJY26} needs a refined estimate
of the projection remainder \cite[Theorem~2]{KJY26} and a separate
argument for the full error.
Here the linearization is taken at the projection $\ztilde_k$ of $z_k$ onto
$\aff\cM$, which lies in $\cM$ and is therefore a fixed point.  The
perturbation $u_k=z_k-\ztilde_k$ is then $P$-orthogonal to $\cT_\cM$, the
linearization contracts it in the normal direction
(\cref{lem:normal-contraction}), and the remainder is $o(\|u_k\|_P)$, so
only the continuous differentiability of $R$ near $\cM$ is needed.

\paragraph{An alternative argument.}
The ADMM error bound of \cite[Theorem~5]{KJY26} gives another proof of
$d_P(Rz) \le \rho \, d_P(z)$, for some $\rho \in (0,1)$, near a strictly
complementary KKT point, and hence of uniform local linear convergence.
ADMM for \eqref{eq:sdp} with penalty parameter~$\tau$ can be written as
the Douglas--Rachford iteration $Z_{k+1}=T_\mathrm{ADMM}(Z_k)$
in the variable $Z=X-\tau S$ \cite[Eq.\ (4)]{KJY26}.
For one PDHG step $(X,y)\mapsto(X^+,y^+)$ with $S=C-\cA^\ast y$
and $Z=X-\tau S$, direct substitution gives
$T_\mathrm{ADMM}(Z)-Z = X^+-X+\sigma^{-1}\cA^\ast(\cA\cA^\ast)^{-1}(y^+-y)$,
so $\|T_\mathrm{ADMM}(Z)-Z\|_\fro\le c\|z-Rz\|_P$ for a constant $c$
depending only on $\tau$, $\sigma$, and~$\cA$.
Under strict complementarity, \cite[Theorem~5]{KJY26} gives
\[
  \dist(Z,\fix(T_\mathrm{ADMM}))\le\kappa\|T_\mathrm{ADMM}(Z)-Z\|_\fro
\]
for~$Z$ near $Z_\star = X_\star-\tau S_\star$.  The Lipschitz map
$Z \mapsto (\proj_{\SV{n}{+}}(Z), \tau^{-1}\proj_{\SV{n}{+}}(-Z))$
sends $\fix(T_\mathrm{ADMM})$ to primal--dual solution pairs and 
maps~$Z$ to $(X^+,S+\tau^{-1}(X^+-X))$, whose second component differs from
$C-\cA^\ast y^+$ by at most a multiple of $\|z-Rz\|_P$.
Hence $d_P(Rz)\le\kappa'\|z-Rz\|_P$, and statement~(b) follows as in the
proof of (a)$\Rightarrow$(b) in \cref{thm:msr}.
This argument, however, gives a contraction factor in terms of the
unspecified constant~$\kappa$ of \cite[Theorem~5]{KJY26}.
The linearization in \cref{lem:normal-contraction} instead relates
the contraction to the linearized PDHG map, which leads to the bound
$\rho_\star$ on the asymptotic rate in \cref{thm:rate} below.

\paragraph{Strict complementarity implies metric subregularity.}
The proof of \cref{thm:sc} and the alternative argument both establish
statement~(b) of \cref{thm:msr} near a strictly complementary KKT point.
Statement~(b) concerns PDHG with particular stepsizes, but by \cref{thm:msr}
it is equivalent to metric subregularity, a property of the SDP alone.
Moreover, \cref{lem:normal-contraction} uses strict complementarity
only at~$z_\star$, not the convergence of a PDHG sequence.
This gives the following corollary relating strict complementarity and
metric subregularity.
\begin{corollary} \label{cor:sc-msr}
Suppose \cref{asp:prob} holds and $z_\star\in\cZ_\star$ satisfies strict
complementarity.  Then $F_\KKT$ is metrically subregular at $z_\star$ for
$0$, and so is $\Ftilde_\KKT$ at $(X_\star,y_\star,S_\star)$.
\end{corollary}
\begin{proof}
Fix any $\tau,\sigma > 0$ with $\tau\sigma\|\cA\|^2<1$.
As in the proof of \cref{thm:sc}, \cref{lem:normal-contraction} gives
\cref{thm:msr}(b), hence statement (a).
The claim for $\Ftilde_\KKT$ follows from~\eqref{eq:kkt-transfer}.
\end{proof}

Whether strict complementarity implies metric subregularity of the KKT
mapping is posed as an open question in \cite[Section~10]{KJY26}.
Ding, Lu, and Yang \cite{DLY26} recently established this implication
for conic programs through a quadratic facial-violation condition.
\Cref{cor:sc-msr} gives an independent proof for SDP, and,
by the alternative argument above,
the implication also follows from \cite[Theorem~5]{KJY26}. 
\Cref{fig:implications} summarizes the relations
between regularity conditions established so far.

\begin{figure}[t]
\centering
\begin{tikzpicture}[
  box/.style={draw, line width=0.6pt, rounded corners=3pt, align=center,
    font=\small, inner xsep=6pt, inner ysep=5pt, minimum height=11mm,
    text width=#1},
  box/.default=3.3cm,
  imp/.style={double, double distance=1.4pt, line width=0.45pt,
    -{Implies}, shorten <=3pt, shorten >=3pt},
  iff/.style={imp, {Implies}-{Implies}},
  lab/.style={font=\footnotesize, inner sep=3pt}
]
\node[box] (nd) at (-2.3,0) {primal--dual\\ nondegeneracy};
\node[box] (sc) at ( 2.3,0) {strict\\ complementarity};
\node[box=6.4cm] (msr) at (0,-2.0) {metric subregularity of $F_\KKT$};
\node[box=6.4cm] (ull) at (0,-3.9)
  {uniform local linear convergence of PDHG};
\draw[imp] (nd.south) -- (nd.south |- msr.north)
  node[lab, midway, left=3pt] {\cref{thm:pdnd}, \cite{CS08,DR04}};
\draw[imp] (sc.south) -- (sc.south |- msr.north)
  node[lab, midway, right=3pt] {\cref{cor:sc-msr}, \cite{DLY26}};
\draw[iff] (msr.south) -- (ull.north)
  node[lab, midway, right=3pt] {\cref{thm:msr}};
\end{tikzpicture}
\caption{Relations among regularity conditions at a KKT point $z_\star$ and
the local convergence of PDHG.  The conditions in the top two rows are
properties of the SDP and $z_\star$ only.  The equivalence holds for
$\tau\sigma\|\cA\|^2 \le 1$ (\cref{thm:msr} and \cref{prop:boundary}).}
\label{fig:implications}
\end{figure}

\subsection{The asymptotic rate}

We now bound the asymptotic rate of PDHG by the spectrum of
the linearized map~$G_\star$.

\begin{theorem}[Asymptotic rate under strict complementarity]
\label{thm:rate}
Suppose \cref{asp:prob,asp:sc} hold, $\tau\sigma\|\cA\|^2<1$,
and $z_k \to z_\star$.
\begin{enumerate}[label=(\alph*)]
\item Let $\rho_\star$ be the largest modulus of the eigenvalues
of $G_\star$ other than one:
\[
  \rho_\star := \max\big\{|\lambda| \;\big|\; \lambda \text{ is an
  eigenvalue of } G_\star,\ \lambda\neq1\big\}.
\]
Then $\rho_\star<1$, and $G_\star^k$ converges to
the $P$-orthogonal projector onto $\cT_\cM$.

\item For every $\rho\in(\rho_\star,1)$ there is $c>0$ such that, for all
$k\in\natint$,
\[
  \dist_P(z_k,\cM)\le c\rho^k,\qquad
  \|z_{k+1}-z_k\|_P\le c\rho^k,\qquad
  \|z_k-z_\star\|_P\le c\rho^k.
\]
\end{enumerate}
\end{theorem}
\begin{proof}
(a) By \eqref{eq:G-fne}, $G_\star$ is nonexpansive in the $P$-norm.  Let
$G_\star^\ast$ be its adjoint with respect to $\inprod{\cdot}{\cdot}_P$.
If $G_\star h=h$, then
\[
\|G_\star^\ast h-h\|_P^2=\|G_\star^\ast h\|_P^2-2\inprod{h}{G_\star h}_P
+\|h\|_P^2=\|G_\star^\ast h\|_P^2-\|h\|_P^2 \le 0,
\]
so $\fix(G_\star)\subseteq\fix(G_\star^\ast)$.
Hence, for $u \perp_P \cT_\cM = \fix(G_\star)$ (\cref{lem:fixG})
and $v \in \cT_\cM$,
\[
\inprod{G_\star u}{v}_P=\inprod{u}{G_\star^\ast v}_P=\inprod{u}{v}_P=0,
\]
so $\cT_\cM^{\perp_P}$ is invariant.  With respect to the decomposition
$\cT_\cM\oplus\cT_\cM^{\perp_P}$, $G_\star$ is block diagonal with blocks
$I$ and $N_\star=G_\star|_{\cT_\cM^{\perp_P}}$.
On~$\cT_\cM^{\perp_P}$, $(I-\proj^P)G_\star=N_\star$, so~\eqref{eq:normal-G}
gives $\|N_\star\|_P < 1$.
Hence the eigenvalues of $G_\star$ other than one are the eigenvalues
of~$N_\star$, so~$\rho_\star$ is the spectral radius of~$N_\star$ 
and $\rho_\star < 1$. Moreover,
$G_\star^k = \proj^P + N_\star^k(I-\proj^P) \to \proj^P$.

(b) For large $k$, let $\ztilde_k = \proj^P_{\aff\cM}(z_k)$ and
$u_k = z_k-\ztilde_k\perp_P\cT_\cM$.  As in the proof of
\cref{lem:normal-contraction}, $\ztilde_k \in \cC_\epsilon$,
$\|u_k\|_P=\dist_P(z_k,\cM)\le\delta$, and
$z_{k+1}=\ztilde_k+G_{\ztilde_k}u_k+r_{\ztilde_k}(u_k)$.  Projecting onto
$\aff\cM=\ztilde_k+\cT_\cM$ and using the invariance in part (a), we obtain
\[
  u_{k+1} = (I-\proj^P) \big(G_{\ztilde_k} u_k + r_{\ztilde_k}(u_k) \big)
  = N_\star u_k + e_k,
\]
where $e_k = (I-\proj^P) \big((G_{\ztilde_k}-G_\star) u_k
+ r_{\ztilde_k}(u_k)\big)$.  By \eqref{eq:normal-G},
\[
  \|e_k\|_P\le\beta_k\|u_k\|_P, \quad \text{with }
  \beta_k = \|G_{\ztilde_k}-G_\star\|_P+\omega(\|u_k\|_P).
\]
Since $z_k\to z_\star$, we have $\ztilde_k\to z_\star$ and $u_k\to0$, so
$\beta_k\to0$ by continuity of $\D R$.  Given $\rho\in(\rho_\star,1)$,
choose a norm $\tnorm{\cdot}$ on $\cT_\cM^{\perp_P}$ and a constant $c_0$ 
satisfying
\[
\tnorm{N_\star u} \le \tfrac{\rho_\star+\rho}{2} \tnorm{u}, \qquad
\tfrac{1}{c_0} \|u\|_P \le \tnorm{u} \le c_0 \|u\|_P.
\]
Then for some $k_0 \in \natint$, $\tnorm{u_{k+1}} \le
\big((\rho_\star+\rho)/2+c_0^2\beta_k\big) \tnorm{u_k}\le\rho\tnorm{u_k}$
for all $k\ge k_0$, so
$\dist_P(z_k,\cM)=\|u_k\|_P\le c_1\rho^k$ for $k\ge k_0$, with
$c_1=c_0^2\rho^{-k_0}\|u_{k_0}\|_P$.
Since $\ztilde_k$ is a fixed point of $R$,
\[
\|z_{k+1}-z_k\|_P \le \|Rz_k-R\ztilde_k\|_P + \|\ztilde_k-z_k\|_P
\le 2\|u_k\|_P \le 2c_1\rho^k,
\]
and since $z_j\to z_\star$,
$\|z_k-z_\star\|_P\le\sum_{j\ge k}\|z_{j+1}-z_j\|_P\le2c_1\rho^k/(1-\rho)$
for $k\ge k_0$.  Hence the three bounds hold for $k\ge k_0$ with
$c=2c_1/(1-\rho)$, and enlarging $c$ makes them hold for the finitely many
$k<k_0$ as well.
\end{proof}

The bound $\rho_\star$ is determined by $\cA$, $\tau$, $\sigma$, and
$(X_\star,S_\star)$ through $G_\star$, and is computable from its
eigenvalues.  If $z_\star$ is an isolated KKT point, then
$\cT_\cM=\{0\}$ and~$\rho_\star$ is the spectral radius of $G_\star$.
The bound $\rho_\star$ applies to every PDHG sequence converging
to~$z_\star$.  By contrast, the error-bound argument in the proof of
\cref{thm:msr} gives the rate $\kappa/\sqrt{1+\kappa^2}$,
which depends on the constant~$\kappa$ in \cref{thm:msr}(a)
and is not computable in general.

\section{A sublinear example without regularity}
\label{sec:sublinear}

Strict complementarity and primal--dual nondegeneracy are sufficient for
local linear convergence (\cref{fig:implications}).  By \cref{thm:msr},
the necessary and sufficient condition is metric subregularity.
We now show that, when it fails, PDHG can converge genuinely sublinearly.

\begin{example} \label{ex:sublinear}
Consider the following primal SDP
\[
\begin{array}{ll}
  \mini & X_{33} \\
  \st   & X_{11}=1, \;\; X_{22} + (X_{13}+X_{31})/2 = 0 \\
		& X \in \SV{3}{+}.
\end{array}
\]
The unique solution is $X_\star = E_{11}$, $y_\star=0$,
and $S_\star=E_{33}$, which fails strict complementarity.
The point is also dual degenerate, because
$E_{12} \in \Nullspace(\cA) \cap \cN_{S_\star}$.
Moreover, $F_\KKT$ is not metrically subregular at $z_\star=(E_{11},0)$
for $0$; see \cref{app:no-msr} for a proof.
By \cref{thm:msr}, PDHG does not converge uniformly locally linearly
at~$z_\star$ for any stepsizes with $\tau\sigma\|\cA\|^2<1$, and
\cref{thm:sublinear} shows that some PDHG sequences converge sublinearly.
\end{example}

\begin{theorem} \label{thm:sublinear}
Consider \cref{ex:sublinear} with $\tau=\sigma=1/2$, so that
$\tau\sigma\|\cA\|^2=3/8$.  There exists $\oline{\alpha}>0$ such that,
for every $\alpha_0\in(0,\oline{\alpha})$, PDHG started from
\begin{equation} \label{eq:sublinear-init}
  X_0 = v_0 v_0^\tran \text{ with } v_0=(1,\alpha_0,-\alpha_0^2), \qquad
  y_0 = (-3\alpha_0^4, \, -2\alpha_0^2)
\end{equation}
generates iterates with $\rank X_k=1$ and $\alpha_k:=(X_k)_{12}$ satisfying
$0<\alpha_{k+1}<\alpha_k$ for all~$k$, and
\[
  \sqrt{2k}\,\alpha_k\to1,\qquad
  \sqrt{k}\,\|X_k-X_\star\|_\fro\to1,\qquad
  k\,\|y_k-y_\star\|_2\to1 .
\]
In particular, neither $\{X_k\}$ nor $\{y_k\}$ converges R-linearly.
\end{theorem}

The proof is given in \cref{app:sublinear}, and \cref{fig:sublinear}
illustrates the behavior predicted by \cref{thm:sublinear}.
The theorem marks a contrast between SDP and LP.
For LP, the KKT mapping is a polyhedral multifunction and therefore
metrically subregular at every KKT point \cite{Rob81}.  By \cref{thm:msr},
PDHG for LP converges uniformly locally linearly near every KKT point,
for all admissible stepsizes; see also \cite{LY24}.
For SDP, the PSD cone is not polyhedral, its projection is not piecewise
affine at singular matrices, and, as \cref{thm:sublinear} shows, the local
linear regime can be lost.

\begin{figure}[t]
\centering
\includegraphics[width=0.65\textwidth]{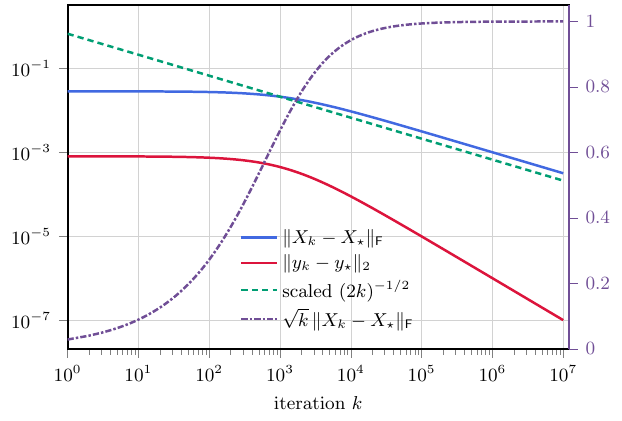}
\caption{Numerical behavior of PDHG on the SDP in \cref{ex:sublinear}
with $\tau=\sigma=1/2$, started from \eqref{eq:sublinear-init}
with $\alpha_0 = 0.02$.
The left vertical axis uses a logarithmic scale and shows the distances
to the KKT point; the dashed curve is a scaled multiple of $(2k)^{-1/2}$.
The right vertical axis uses a linear scale and shows
$\sqrt{k}\,\|X_k-X_\star\|_\fro$, which approaches one,
consistent with \cref{thm:sublinear}.}
\label{fig:sublinear}
\end{figure}

\section{Rank identification and partial smoothness} \label{sec:rank-id}

Partial smoothness \cite{Lewis02} and identifiability \cite{HL04,DL14}
explain why proximal-type methods reach the active manifold of a solution
in finitely many iterations, and Liang, Fadili, and Peyr\'e \cite{LFP18}
combine identification with a local linearization for primal--dual
splitting methods, which include PDHG.  For SDP, the active manifold
consists of the PSD matrices whose rank equals that of the solution.
This section proves rank identification for PDHG directly under strict
complementarity, and compares our convergence analysis with \cite{LFP18}.

\begin{proposition}[Finite-time rank identification] \label{prop:rank-id}
Suppose \cref{asp:prob,asp:sc} hold, and let $r = \rank X_\star$.  Define
\[
  Z_k = X_k-\tau(C-\cA^\ast y_k) = X_k-\tau S_k, \qquad
  \Shat_{k+1} = \tau^{-1} (X_{k+1} - Z_k).
\]
Then, for all $k \in \natint$, $X_{k+1} = \proj_{\SV{n}{+}} (Z_k)$,
$\Shat_{k+1} \succeq 0$, and $\inprod{X_{k+1}}{\Shat_{k+1}} = 0$.
Moreover, there exists an integer $k_\ID$ such that, for all $k \ge k_\ID$,
\[
  \inertia (Z_k) = (r,n-r,0), \qquad \rank(X_{k+1}) = r, \qquad
  \rank(\Shat_{k+1}) = n-r.
\]
\end{proposition}
\begin{proof}
If $Z_k = Q \diag(\lambda_1,\ldots,\lambda_n) Q^\tran$, then
\[
  X_{k+1} = Q \diag(\lambda_1^+,\ldots,\lambda_n^+) Q^\tran, \qquad
  \Shat_{k+1} = \tau^{-1}
  Q \diag((-\lambda_1)^+,\ldots,(-\lambda_n)^+) Q^\tran,
\]
where $t^+ = \max\{t,0\}$; this gives the claims for all $k$.
Since $(X_k,y_k,S_k)\to(X_\star,y_\star,S_\star)$, we have $Z_k\to Z_\star$.
The eigenvalues of $Z_\star=X_\star-\tau S_\star$ are the $r$ positive
eigenvalues of $X_\star$ and $-\tau$ times the $n-r$ positive eigenvalues
of $S_\star$; one of the two groups is empty if $r=0$ or $r=n$.  Let
$\gamma>0$ be the smallest absolute value of an eigenvalue of $Z_\star$,
which is positive by strict complementarity.  For sufficiently large $k$,
$\|Z_k-Z_\star\|_2<\gamma$, and Weyl's inequality gives
$\inertia(Z_k) = (r,n-r,0)$; the rank claims then follow from the display.
\end{proof}

The proof only uses that $Z_\star$ is nonsingular, which for complementary
$X_\star, S_\star \succeq 0$ is equivalent to strict complementarity.
An analogous result for ADMM is \cite[Proposition~3]{KJY26}.
The matrix $\Shat_{k+1}=S_k+\tau^{-1}(X_{k+1}-X_k)$ is a dual
slack generated by the projection step, and it differs from the actual
dual slack $S_k=C-\cA^\ast y_k$.  The latter need not be positive semidefinite,
and its rank need not stabilize, as the following example shows.

\begin{example} \label{ex:rank-id}
Consider the SDP
\[
\begin{array}{ll}
  \mini & 0 \\
  \st   & X_{11}=1, \;\; (X_{12}+X_{21})/2=0, \;\; X_{22}=1 \\
		& X \in \SV{2}{+}.
\end{array}
\]
The unique KKT point is $X_\star=I_2$, $y_\star=0$, $S_\star=0$, which is
strictly complementary.  Take $\tau=1$ and $\sigma=3/4$, so that
$\tau \sigma \|\cA\|^2 = 3/4$.  Initialize at
$X_0 = \diag (1 - \epsilon /\sqrt{3}, 1)$ and $y_0 = 0$,
with a small $\epsilon>0$.
The iterates remain in the interior of~$\SV{2}{+}$ and have the form
$X_k = \diag(1+\alpha_k, 1)$ and $y_k = (\beta_k,0,0)$,
where $\alpha_{k+1} = \alpha_k + \beta_k$ and
$\beta_{k+1} = -\tfrac{3}{4} \alpha_k - \tfrac{1}{2} \beta_k$,
with $\alpha_0 = -\epsilon/\sqrt{3}$ and $\beta_0 = 0$.
Eliminating $\alpha_k$ gives
\[
  \beta_0 = 0, \qquad \beta_1 = \tfrac{\sqrt{3}}{4}\epsilon, \qquad
  \beta_{k+2} - \tfrac{1}{2} \beta_{k+1} + \tfrac{1}{4} \beta_k = 0
  \text{ for $k \in \natint$.}
\] 
Direct calculation shows that $\beta_k = \epsilon 2^{-k} \sin (k\pi/3)$
and $S_k = C - \cA^\ast y_k = -\beta_k E_{11}$.
So $S_k=0$ for $k=0,3,6,\ldots$, while $S_k \neq 0$ for infinitely many
other $k$.  The rank of $S_k$ never stabilizes at $\rank S_\star = 0$.
\end{example}

\Cref{prop:rank-id} shows that $X_{k+1}$ eventually lies in the manifold
$\cM_r = \{X \in \SV{n}{+} \mid \rank X=r\}$.  This manifold is curved, and
identifying it does not mean identifying the face
$\{X \succeq 0 \mid \Range X \subseteq \Range X_\star\}$:
the eigenspace of the $r$ positive eigenvalues of~$X_k$ may change
at every iteration and reach $\Range X_\star$ only in the limit.
The manifold~$\cM_r$ should not be confused with the set~$\cM$ 
in~\eqref{eq:M}: $\cM$ is a relatively open subset of an affine set and
represents the KKT set near a strictly complementary point.
Identification concerns $\cM_r$,
whereas the convergence analysis concerns~$\cM$.

\paragraph{Relation to partial smoothness.} \label{sec:lfp}
A function is partly smooth relative to a manifold if, roughly, it is
smooth along the manifold and sharp across it.  Under a nondegeneracy
condition, such a manifold is identifiable: sequences converging to the
solution with suitably convergent subgradients, such as the iterates of
proximal-type methods, eventually lie in it.
Liang, Fadili, and Peyr\'e \cite{LFP18} specialize this theory to
primal--dual splitting, which covers PDHG~\eqref{eq:pdhg}.
The function $\inprod{C}{X}+\delta_{\SV{n}{+}}(X)$ is partly smooth
at~$X_\star$ relative to $\cM_r$, their nondegeneracy condition becomes
$-S_\star\in\ri N_{\SV{n}{+}}(X_\star)$, which is strict complementarity,
and their identification result yields $X_k \in \cM_r$ for all large $k$,
which is the rank statement in \cref{prop:rank-id}.

The local analysis in \cite{LFP18} linearizes the iteration at the limit
point.  After a sign change of the dual variables, their linearized operator
$M_{\mathrm{PD}}$ agrees with $G_\star$ on $\cT_{X_\star} \times \R[m]$,
the tangent space of the identified manifold $\cM_r \times \R[m]$,
and the fixed space of $M_{\mathrm{PD}}$ in this subspace is
$\fix(G_\star) = \cT_\cM$ (\cref{lem:fixG}).
On $\cN_{X_\star} \times \{0\}$, $M_{\mathrm{PD}}$ is the identity,
whereas $G_\star$ maps it into $\{0\}\times\R[m]$.
Hence their rate factor $\rho(M_{\mathrm{PD}}-M_{\mathrm{PD}}^\infty)$,
with $M_{\mathrm{PD}}^\infty=\lim_k M_{\mathrm{PD}}^k$, equals $\rho_\star$.
In \cite[Theorem~3.7]{LFP18}, linear convergence is established for 
the error component $(I-M_{\mathrm{PD}}^\infty)(z_k-z_\star)$ transverse to
$\fix(M_{\mathrm{PD}})$; when the functions are locally polyhedral around
the solution, the same theorem gives linear convergence of the full error.
Our analysis complements this framework by controlling the full
error in the nonpolyhedral SDP setting.
The difference lies in the error component in~$\cT_\cM$,
which vanishes when the KKT point is isolated.
Specifically, the KKT set is locally affine under strict complementarity
(\cref{lem:local-geometry}), and, under the linearization at the projection
onto $\aff\cM$ (\cref{sec:admm}), the error component normal to it contracts
(\cref{lem:normal-contraction}).
Together with firm nonexpansiveness, these properties yield linear
convergence of the full error and of the distance to the KKT set
at the same rate (\cref{thm:rate}).

\section{Numerical experiments} \label{sec:numerics}

We present numerical evidence supporting the theoretical findings.
The experiments are designed to illustrate the following.
\begin{enumerate}
\item Local (R-)linear convergence is observed under different combinations
of regularity conditions.

\item Under strict complementarity (SC), the observed asymptotic rate agrees
with the bound $\rho_\star$ of \cref{thm:rate}, computed from the
eigenvalues of~$G_\star$ at the final iterate.

\item When SC is close to failure, PDHG applied to SDP may converge
extremely slowly, and no clear linear regime may be observed
within the prescribed computational budget.
\end{enumerate}
All numerical experiments were conducted on a Mac mini equipped with
an Apple M4 Pro chip and 48 GB of unified memory, running macOS 15.7.4.
The algorithms were implemented in Python.
The code and synthetic SDP data used in the experiments are available at
\url{https://github.com/NumOptLab/splitting-sdp/tree/pdhg-sdp-v1.0}.

For the primal--dual pair of SDPs in \eqref{eq:sdp},
we define primal infeasibility $\epsilon_\mathrm{p}$,
dual infeasibility $\epsilon_\mathrm{d}$,
complementarity residual $\epsilon_\mathrm{c}$,
and relative objective gap $\epsilon_\mathrm{gap}$ by
\begin{align*}
  \epsilon_\mathrm{p} &= \frac{\|\cA X-b\|_2}{1 + \|b\|_2}, \;\; &
  \epsilon_\mathrm{d} &= \max\left\{0,
	\frac{-\lambda_\mathrm{min}(S)}{1 + \|C\|_\fro} \right\}, \\
  \epsilon_\mathrm{c} &= \frac{\inprod{X}{\proj_{\SV{n}{+}}(S)}}%
	{1+|\inprod{C}{X}|+|\inprod{b}{y}|}, \;\; &
  \epsilon_\mathrm{gap} &= \frac{|\inprod{C}{X}-b^\tran y|}%
	{1+|\inprod{C}{X}|+|\inprod{b}{y}|}.
\end{align*}
The overall KKT residual is then defined as $\epsilon_\mathrm{KKT} := \max
\{ \epsilon_\mathrm{p}, \epsilon_\mathrm{d}, \epsilon_\mathrm{c},
\epsilon_\mathrm{gap} \}$.  The algorithm is terminated
when $\epsilon_\mathrm{KKT} \le 10^{-8}$
or when the iteration count reaches $10^5$.
The SDP instances used in the experiments are summarized in \cref{tab:data}.
Strict complementarity is assessed numerically by computing
the smallest absolute eigenvalue of the final $Z$-iterate
(where $Z_k=X_k-\tau S_k$),
and for small instances, nondegeneracy is checked numerically
following approaches in \cite[\S8.2]{KJY26}.
\begin{table}[t]
\centering
\caption{Basic information on the tested SDP instances.}
\label{tab:data}
\vspace{1.0ex}
\small
\begin{adjustbox}{max width=\textwidth}
\begin{tabular}{lrrrrr@{\qquad\quad}lrrrrr} \toprule
Instance & $n$ & $m$ & $\tau$ & $\sigma$
& $\tau\sigma\|\cA\|^2$
& Instance & $n$ & $m$ & $\tau$ & $\sigma$
& $\tau\sigma\|\cA\|^2$ \\ \midrule
\texttt{maxG*}         & 800 &    800 & 1.0000 & 1.0000 & 1.0000
& \texttt{XM-48}       & 144 &    241 & 0.5715 & 0.5715 & 0.9798 \\
\texttt{QS-20}         & 231 &  16402 & 0.2135 & 0.2135 & 0.9801
& \texttt{XM-93}       & 279 &    466 & 0.5715 & 0.5715 & 0.9798 \\
\texttt{QS-30}         & 496 &  77377 & 0.0558 & 0.5578 & 0.9805
& \texttt{XM-149}      & 447 &    746 & 0.5715 & 0.5715 & 0.9798 \\
\texttt{QS-40}         & 861 & 236202 & 0.0486 & 0.4859 & 0.9800
& \texttt{rose13}      & 105 &   2379 & 0.3000 & 0.3000 & 0.5400 \\
\texttt{pdnd-nosc-100} & 100 &    136 & 0.5000 & 0.5000 & 0.5313
& \texttt{cnhil10}     & 220 &   5005 & 0.2214 & 0.2214 & 0.9804 \\
\texttt{pdnd-nosc-200} & 200 &    210 & 0.5000 & 0.5000 & 0.5313
& \texttt{BQP-r1-20-1} & 231 &  20601 & 0.3000 & 0.3000 & 0.9967 \\
\texttt{pdnd-nosc-500} & 500 &    378 & 0.5000 & 0.5000 & 0.5313
& \texttt{msr-only}    &   4 &      5 & 0.4500 & 0.4500 & 0.8100 \\ \bottomrule
\end{tabular}
\end{adjustbox}
\end{table}

\subsection{Demonstration of local linear convergence}

We first report SDP instances for which some regularity conditions in our
analysis appear to hold numerically.  The test set contains examples from
several application domains, together with synthetic instances.
In all these experiments, PDHG exhibits a clear local linear convergence
regime after a transient phase.

\paragraph{MaxCut SDP: SC holds and primal--dual ND holds.}
\Cref{fig:maxcut}%
\begin{figure}[tb]
\centering
\includegraphics[width=\textwidth]{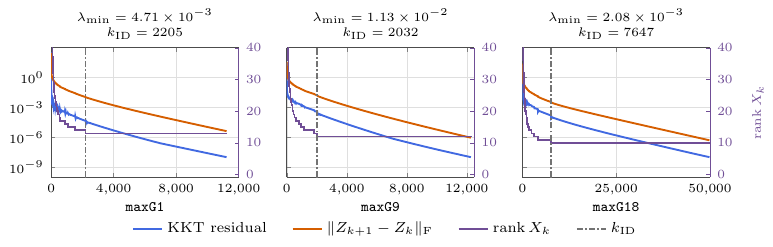}

\caption{PDHG convergence on MaxCut SDP instances.
The blue and orange curves report the KKT residual and
$\|Z_{k+1}-Z_k\|_\fro$, respectively, on a logarithmic scale; the purple curve shows
$\rank X_k$ on the right axis in the standard scale.  The dash-dotted 
vertical line marks the rank-identification iteration $k_\ID$,
and $\lambda_\mathrm{min}$ denotes the smallest absolute eigenvalue 
of the final $Z$-iterate.}
\label{fig:maxcut}
\end{figure}
reports three representative MaxCut SDP instances, with PDHG started
at the origin.
For these instances, $\cA X=\diag(X)$, so $\|\cA\|=1$ and
$\tau\sigma\|\cA\|^2=1$; this degenerate case is covered by
\cref{prop:boundary,thm:pdnd},
while the rate bound of \cref{thm:rate} does not apply.
For all three cases, strict complementarity and primal--dual nondegeneracy 
appear to hold numerically.  Consistent with the theory, PDHG enters
a local linear convergence regime after a short transient phase.

The three panels show the evolution of the KKT residual, the difference
$\|Z_{k+1}-Z_k\|_\fro$ of consecutive $Z$-iterates, and the rank of $X_k$.
In each case, the residual and $\|Z_{k+1}-Z_k\|_\fro$ eventually
decrease at an approximately linear rate on the semilogarithmic scale.
The rank of $X_k$ also stabilizes after finitely many iterations.
The reported values of $\lambda_\mathrm{min}$ are bounded away from zero,
which is consistent with strict complementarity, while $k_\ID$ indicates
that rank stabilization occurs well before high-accuracy termination.

\paragraph{Quartic function over sphere (QS):
SC holds and primal ND fails.}
This testbed arises from a classical polynomial optimization problem
and its second-order SDP relaxation; see, \eg, \cite{WH25,YLCT23}.
For the tested instances, the primal solution appears to be unique
and of rank one.  Primal nondegeneracy then fails.
These instances therefore appear to satisfy the regularity assumption
of \cref{thm:sc} but not that of \cref{thm:pdnd}:
they test the regime that is covered by our analysis under strict
complementarity and not by the classical theory based on nondegeneracy.

\Cref{fig:qs}%
\begin{figure}[tb]
\centering
\includegraphics[width=\textwidth]{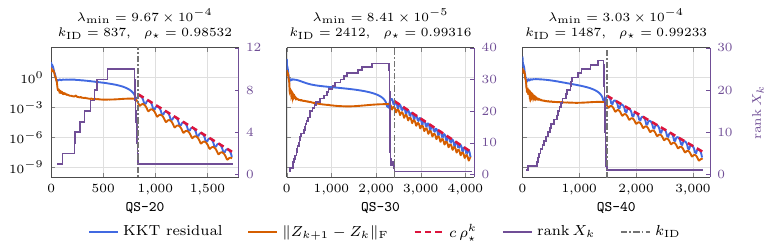}

\caption{PDHG convergence on random QS problems.
The blue and orange curves report the KKT residual and
$\|Z_{k+1}-Z_k\|_\fro$, respectively, on a logarithmic scale;
the purple curve shows $\rank X_k$ on the right axis in the standard scale.
The dash-dotted vertical line marks the rank-identification iteration
$k_\ID$, and $\lambda_\mathrm{min}$ denotes the smallest absolute eigenvalue
of the final $Z$-iterate.
The red dashed line is $c\rho_\star^k$, with $\rho_\star$ from
\cref{thm:rate} evaluated at the final iterate and $c$ the smallest constant
for which the line lies above both curves for $k\ge k_\ID$.}
\label{fig:qs}
\end{figure}
reports three representative instances, with PDHG started at the
origin.
In all three cases, strict complementarity appears to hold numerically,
as indicated by the smallest absolute eigenvalue reported above the panels.
Despite the failure of primal nondegeneracy, both the KKT residual and
$\|Z_{k+1}-Z_k\|_\fro$ eventually decrease R-linearly.
The rank of $X_k$ is also identified after a finite number of iterations,
after which the iterates enter a clear local linear convergence regime.
Eventually, both curves decay at a rate close to the bound~$\rho_\star$ of
\cref{thm:rate}: the dashed line $c\rho_\star^k$,
which by construction lies above both curves, touches them repeatedly.

\paragraph{Synthetic instances: SC fails and primal--dual ND holds.}
\Cref{fig:pdnd-nosc}%
\begin{figure}[tb]
\centering
\includegraphics[width=\textwidth]{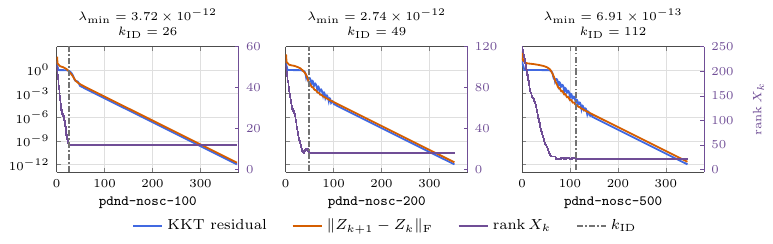}

\caption{PDHG convergence on synthetic SDP instances.
The unique KKT point satisfies primal--dual nondegeneracy and fails
strict complementarity.  The blue and orange curves report the KKT residual
and $\|Z_{k+1}-Z_k\|_\fro$, respectively, on a logarithmic scale;
the purple curve shows $\rank X_k$ on the right axis in the standard scale.
The dash-dotted vertical line marks the rank-identification iteration
$k_\ID$, and $\lambda_\mathrm{min}$ denotes the smallest absolute 
eigenvalue of the final $Z$-iterate.}
\label{fig:pdnd-nosc}
\end{figure}
reports three synthetic SDP instances constructed so that
the optimal solution is unique and satisfies primal--dual nondegeneracy,
while strict complementarity fails.
PDHG is started from a random initial point.
For these instances, PDHG is run to the tighter tolerance
$\epsilon_\mathrm{KKT}\le10^{-12}$, and the values of $\lambda_\mathrm{min}$
are of the order of this tolerance, which reflects the failure of strict
complementarity.
Despite the failure of strict complementarity, PDHG exhibits a clear local
linear convergence regime in all three cases.
The rank of~$X_k$ is identified after a short transient phase,
and thereafter both the KKT residual and $\|Z_{k+1}-Z_k\|_\fro$
decay R-linearly.
These examples illustrate that strict complementarity is not necessary for
the local linear behavior observed in practice,
provided that the primal--dual nondegeneracy condition holds.

\paragraph{Structure-from-motion problems: SC holds.}
\Cref{fig:xm}%
\begin{figure}[tb]
\centering
\includegraphics[width=\textwidth]{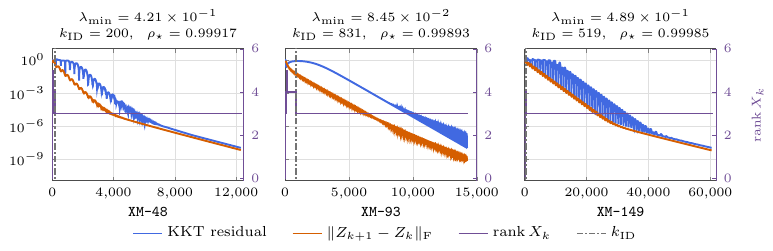}

\caption{PDHG convergence on structure-from-motion problems.
The blue and orange curves report the KKT residual and
$\|Z_{k+1}-Z_k\|_\fro$, respectively, on a logarithmic scale;
the purple curve shows $\rank X_k$ on the right axis in the standard scale.
The dash-dotted vertical line marks the rank-identification iteration
$k_\ID$, and $\lambda_\mathrm{min}$ denotes the smallest absolute
eigenvalue of the final $Z$-iterate.
The value $\rho_\star$ is the rate bound in \cref{thm:rate},
evaluated at the final iterate.}
\label{fig:xm}
\end{figure}
reports three SDP instances from a structure-from-motion testbed, with
PDHG started at the origin.
For these instances, strict complementarity appears to hold numerically.
Verifying the relevant nondegeneracy conditions is computationally
expensive for this testbed,
so we do not report a nondegeneracy certificate.
Nevertheless, the observed behavior is consistent with the preceding
experiments: after a transient phase, (R-)linear convergence is observed,
and the numerical rank of~$X_k$ stabilizes after finitely many iterations.
The reported values of $\rho_\star$ agree closely with the observed rates.
Plotted as in \cref{fig:qs}, the line $c\rho_\star^k$ would nearly coincide
with the curves, so we omit it.

\paragraph{A synthetic instance: SC and ND fail.}
Strict complementarity and primal--dual nondegeneracy are sufficient for
metric subregularity (\cref{fig:implications}).
To see whether PDHG can converge linearly when both fail, we constructed
a small SDP instance, \texttt{msr-only}, with a unique KKT point.
At this point, we check numerically that both strict complementarity and
primal--dual nondegeneracy fail.

To test metric subregularity numerically, we use the error bound
\begin{equation} \label{eq:eb-test}
  \|z-z_\star\|_P \le \kappa\,\|z-Rz\|_P
  \qquad \text{for all } z\in\cB^P_\epsilon(z_\star),
\end{equation}
which is equivalent to metric subregularity of $F_\KKT$ at $z_\star$ for~$0$
(from \cite{vTPVZ25} and \cref{thm:msr}).
Thus metric subregularity holds if and only if the ratio
\[
  q(t) = \min \big\{\|z-Rz\|_P/t \mid \|z-z_\star\|_P =t \big\}
\]
is bounded away from zero for all small $t>0$.  We estimate $q(t)$ for
$t\in [10^{-4},10^{-1}]$ by random sampling, followed by local minimization.
A value bounded away from zero then serves as numerical evidence of
metric subregularity.
\Cref{fig:msr-only}(b)%
\begin{figure}[tb]
\centering
\includegraphics[width=0.97\textwidth]{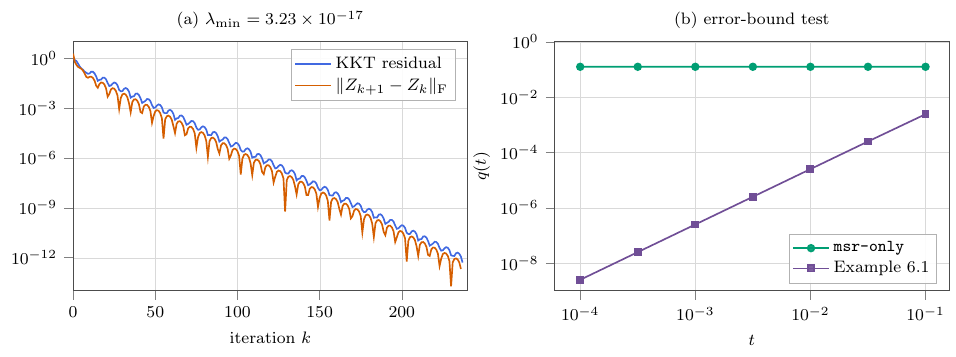}
\caption{The \textup{\texttt{msr-only}} instance.
(a)~Convergence of PDHG; $\lambda_\mathrm{min}$ denotes the smallest
absolute eigenvalue of the final $Z$-iterate.  (b)~Estimates of the ratio
$q(t)$ for \textup{\texttt{msr-only}} and for \cref{ex:sublinear};
values bounded away from zero are evidence of metric subregularity.}
\label{fig:msr-only}
\end{figure}
shows that $q(t)\approx0.124$ throughout this range.  For comparison,
the same procedure applied to \cref{ex:sublinear} with $\tau=\sigma=1/2$,
at which metric subregularity fails, gives $q(t)\approx t^2/4$.  This
agrees with \cref{app:no-msr}, where $\dist(0,F_\KKT(z))=O(\alpha^3)$
along a curve with $\|z-z_\star\|\ge\sqrt2\,\alpha$.

\Cref{fig:msr-only}(a) presents the convergence of PDHG for this instance
from a random initial point, with tolerance
$\epsilon_\mathrm{KKT} \le 10^{-12}$; PDHG converges at a linear rate.
Over $20$ random initial points, the observed rate of $\|z_k-z_\star\|_P$
lies between $0.8927$ and $0.8930$,
consistent with the uniform local linear convergence in \cref{thm:msr}(c).
These observations indicate that metric subregularity, and with it the local
linear convergence of PDHG, can hold when strict complementarity and
primal--dual nondegeneracy both fail.

\subsection{Slow convergence near failure of strict complementarity}

\Cref{fig:stall}%
\begin{figure}[t]
\centering
\includegraphics[width=\textwidth]{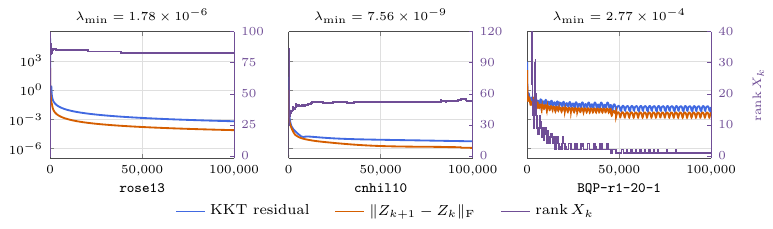}

\caption{PDHG convergence on additional SDP instances.
The blue and orange curves report the KKT residual and
$\|Z_{k+1}-Z_k\|_\fro$, respectively, on a logarithmic scale; the purple curve shows
$\rank X_k$ on the right axis in the standard scale.
Here $\lambda_{\mathrm{min}}$ denotes the smallest absolute eigenvalue of
the final $Z$-iterate.  For these SDPs, PDHG fails to achieve the accuracy
$10^{-8}$ within the stated budget, and no clear linear convergence
is observed.}
\label{fig:stall}
\end{figure}
reports several SDP instances for which PDHG does not
reach the prescribed accuracy within the computational budget,
and for which no clear linear convergence regime is observed.
PDHG is started at the origin for \texttt{rose13} and \texttt{cnhil10},
and from a random initial point for \texttt{BQP-r1-20-1}.
For \texttt{rose13} and \texttt{cnhil10}, the smallest absolute eigenvalue
of the final $Z$-iterate is small, which suggests that these instances
are close to violating strict complementarity.

The instance \texttt{BQP-r1-20-1} contrasts with the numerical results
in~\cite{KJY26}, where ADMM exhibits a clear linear regime on BQP instances
of this type, although strict complementarity and nondegeneracy appear to
fail.  If the linear regime of ADMM reflects metric subregularity at its
limit, then every PDHG sequence converging to that limit eventually
converges linearly (\cref{thm:msr}). Thus, either PDHG and ADMM converge to
different KKT points, or PDHG has not yet reached its local regime,
or its local rate is too close to one for the linear regime to be visible
within $10^5$ iterations.
The final KKT residual of PDHG is still about $2\times10^{-2}$,
so we cannot distinguish between the last two explanations;
for the same reason, $\lambda_\mathrm{min}$ of the final $Z$-iterate
is not informative for this instance.

\section{Conclusion} \label{sec:conclusion}

We studied the local convergence of PDHG for SDP.
Metric subregularity of the KKT mapping, a property of the SDP alone,
characterizes uniform local linear convergence of PDHG for all stepsizes
with $\tau\sigma\|\cA\|^2\le1$.
If PDHG converges to a strictly complementary KKT point, then it eventually
converges linearly, even when the KKT set is not a singleton.
Moreover, for $\tau\sigma\|\cA\|^2<1$, its asymptotic rate is at most
the largest modulus of the eigenvalues (other than one)
of the linearized PDHG map.  As a byproduct of the analysis,
strict complementarity at a KKT point implies metric subregularity
at that point (see also \cite{DLY26}).
Without metric subregularity, PDHG can converge sublinearly:
for a $3 \times 3$ SDP, we proved the rates $\Theta(k^{-1/2})$
in the primal variable and $\Theta(k^{-1})$ in the dual variable,
whereas PDHG always enters a local linear regime for LP.
Numerical experiments support the theory and show slow convergence
near failure of strict complementarity, as observed for LP \cite{LY24}.

Several questions remain open.  First, PDHG with larger stepsizes
or other extrapolation parameters also converges \cite{BUG26,Upadhyaya26}.
Extending the analysis to those settings requires suitable replacements
for the metric and firm-nonexpansiveness arguments used here.
Second, our results are local and concern the limit point selected by
the algorithm.  How primal--dual methods select a KKT point when
the KKT set is not a singleton, and how this selection interacts with
the geometry of the PSD cone, remain to be understood.

\appendix

\section{Proofs for \texorpdfstring{\cref{sec:sublinear}}{Section 6}}
\label{app:sublinear-proofs}

This appendix contains the proofs omitted from \cref{sec:sublinear}.
Throughout, the data are those of \cref{ex:sublinear}.

\subsection{Failure of metric subregularity} \label{app:no-msr}
We first verify that the KKT point is unique.  The primal and dual
optimal values are zero, attained by $X=E_{11}$ and $y=0$.
If $X\succeq0$ is optimal, then $X_{33}=0$ forces $X_{13}=0$.
Then the second constraint gives $X_{22}=0$, and hence $X=E_{11}$.
At a dual optimum, $y_1=b^\tran y=0$, and the slack
$S=C-\cA^\ast y$ has $(1,1)$ entry $-y_1=0$ and $(1,3)$ entry $-y_2/2$, so
$S\succeq0$ forces $y_2=0$.  Thus $\cZ_\star=\{z_\star\}$ with
$z_\star=(E_{11},0)$ and $S_\star=E_{33}$.

We now show that $F_\KKT$ is not metrically subregular at $z_\star$ for~$0$.
For small $\alpha>0$, let $z(\alpha)=(X(\alpha),y(\alpha))$ be the point
\eqref{eq:sublinear-init} with $\alpha_0 \gets \alpha$,
and $v=(1,\alpha,-\alpha^2)$.  Then $\cA X(\alpha) = b$, and
\[
  S(\alpha) = C-\cA^\ast y(\alpha) = \begin{bmatrix}
  3\alpha^4 & 0 & \alpha^2 \\ 0 & 2\alpha^2 & 0 \\ \alpha^2 & 0 & 1
  \end{bmatrix} \succeq 0, \qquad
  S(\alpha)v = (2\alpha^4,\,2\alpha^3,\,0).
\]
Let $\Pi_v=vv^\tran/\|v\|^2$ and $W=(I-\Pi_v)S(\alpha)(I-\Pi_v)$.
Then $W \succeq 0$, $WX(\alpha)=0$,
\[
  -W\in N_{\SV{3}{+}}(X(\alpha)), \qquad \text{and} \qquad
  (S(\alpha)-W, \, 0) \in F_\KKT(z(\alpha)).
\]
Moreover, the norm of
$S(\alpha)-W = \Pi_v S(\alpha) + S(\alpha) \Pi_v - \Pi_vS(\alpha)\Pi_v$
is at most $3\|S(\alpha)v\|/\|v\| = O(\alpha^3)$.
Since $z(\alpha)\to z_\star$, $\cZ_\star=\{z_\star\}$, and
$\|z(\alpha)-z_\star\|\ge\|X(\alpha)-E_{11}\|_\fro\ge\sqrt2\,\alpha$,
the ratio of $\dist(z(\alpha),\cZ_\star)$ to $\dist(0,F_\KKT(z(\alpha)))$
is unbounded as $\alpha\downarrow0$.

\subsection{Proof of \texorpdfstring{\cref{thm:sublinear}}{Theorem 6.2}}
\label{app:sublinear}

Throughout this subsection, $\tau=\sigma=1/2$.
The proof parametrizes rank-one $X$ and $y$ near $z_\star$ by
$\alpha = X_{12} \in \R$ and $w \in \R[4]$, and replaces the PDHG iteration 
by an analytic map in~$(\alpha,w)$.
It then shows that the iterates started from~\eqref{eq:sublinear-init}
stay in a thin tube around a curve; on this tube, the map coincides
with PDHG and $\alpha^+=\alpha-\alpha^3+O(\alpha^4)$.
The rates in \cref{thm:sublinear} then follow from this scalar recursion.
Each $O(\cdot)$ term below is bounded in norm by a constant
(independent of~$k$) times its argument, for all $(\alpha,w)$ in a fixed
neighborhood of the origin.

\paragraph{An analytic auxiliary map.}
At the KKT point, $Z_\star = X_\star-\tau S_\star = \diag(1,0,-1/2)$.
For $Z$ near $Z_\star$, the largest eigenvalue $\lambda_1(Z)$ is simple;
let $P_1(Z)$ be the orthogonal projector onto its eigenspace.
Then $\Pihat(Z) = \lambda_1(Z)P_1(Z)$ is real analytic near $Z_\star$,
and $\Pihat(Z) = \proj_{\SV{3}{+}}(Z)$ whenever the other two eigenvalues
of~$Z$ are nonpositive.

We now define a real-analytic map $\psi \colon (\alpha,w) \mapsto
(\alpha^+,w^+)$ on a neighborhood of the origin.
It is the PDHG map~\eqref{eq:pdhg} with $\proj_{\SV{3}{+}}$ replaced
by~$\Pihat$.  Since $Z_\star$ is singular, $\proj_{\SV{3}{+}}$ is not
differentiable at~$Z_\star$, so we expand $\psi$ instead and show in
\cref{lem:tube} that it coincides with PDHG on a tube $\cQ$, defined below,
that contains the iterates.

For $w=(w_1,w_2,w_3,w_4)$ with $w_1>-1$, let $y=(w_2,w_4)$ and
\[
  X = \frac{1}{1+w_1} \begin{bmatrix} 1+w_1 \\ \alpha \\ w_3 \end{bmatrix} 
  \begin{bmatrix} 1+w_1 \\ \alpha \\ w_3 \end{bmatrix}^\tran
  = \begin{bmatrix}
	1+w_1 & \alpha & w_3 \\
    \alpha & \alpha^2/(1+w_1) & \alpha w_3/(1+w_1) \\
	w_3 & \alpha w_3/(1+w_1) & w_3^2/(1+w_1)
  \end{bmatrix}.
\]
Every rank-one $X\in\SV{3}{+}$ with $X_{11}>0$ has this form, and
$(\alpha,w)=(0,0)$ gives the KKT point $(X_\star,y_\star)=(E_{11},0)$.
Then
\[
  Z = X-\tau(C-\cA^\ast y) = \begin{bmatrix}
	1+w_1+w_2/2 & \alpha & w_3+w_4/4 \\
	\alpha & \alpha^2/(1+w_1)+w_4/2 & \alpha w_3/(1+w_1) \\
	w_3+w_4/4 & \alpha w_3/(1+w_1) & w_3^2/(1+w_1)-1/2
\end{bmatrix},
\]
which is close to $Z_\star$ when $(\alpha,w)$ is close to the origin.
Let $\lambda=\lambda_1(Z)$, with eigenvector normalized as $v=(1,v_2,v_3)$.
Then $\lambda$, $v_2$, and $v_3$ are real analytic in $(\alpha,w)$ and
equal to $(1,0,0)$ at the origin, and $Zv=\lambda v$ gives
\[
  \lambda = Z_{11}+\alpha v_2+Z_{13}v_3 ,\qquad
  v_2 = \frac{\alpha+Z_{23}v_3}{\lambda-Z_{22}}, \qquad
  v_3 = \frac{Z_{13}+Z_{23}v_2}{\lambda-Z_{33}}.
\]

With $\proj_{\SV{3}{+}}$ replaced by~$\Pihat$, the PDHG
step~\eqref{eq:pdhg} becomes
\[
  X^+ = \Pihat(Z) = \tfrac{\lambda}{\|v\|^2} \, vv^\tran, \qquad
  y^+ = y + \sigma \big(b-\cA(2X^+-X)\big).
\]
The matrix $X^+$ has the same form as~$X$,
with $1+w_1^+ = {\lambda}/{\|v\|^2}$, $\alpha^+ = (1+w_1^+)v_2$,
$w_3^+ = (1+w_1^+)v_3$, and,
since $\cA X=(X_{11},\,X_{22}+X_{13})$ and $b=(1,0)$,
\[
  w_2^+ = w_2 + \frac12 w_1 - w_1^+, \qquad
  w_4^+ = w_4 + \frac12 \Big(\frac{\alpha^2}{1+w_1} + w_3\Big)
  - \Big(\frac{(\alpha^+)^2}{1+w_1^+}+w_3^+\Big).
\]

\paragraph{Expansions.}
We now expand~$\psi$ around the curve $w=\phi(\alpha)$, where
$\phi(\alpha) = (0,0,-\alpha^2,-2\alpha^2)$.  The function $\psi$ maps
this curve into itself up to terms of order~$\alpha^4$, and along it,
$\alpha$ decreases by~$\alpha^3$ per step to leading order.  The deviation
$u = w-\phi(\alpha)$ from the curve is mapped approximately to~$Bu$, where
\[
  B = \begin{bmatrix} B_1 & 0 \\ 0 & B_2 \end{bmatrix}, \quad
  B_1 = \frac12 \begin{bmatrix} 2 & 1 \\ -1 & 1 \end{bmatrix}, \quad
  B_2 = \frac16 \begin{bmatrix} 4 & 1 \\ -1 & 5 \end{bmatrix}
\]
is the derivative of~$w^+$ with respect to~$w$ at the origin.
The spectral radius of~$B$ is $\sqrt3/2$,
so the deviation decays geometrically, up to higher-order terms.
\Cref{lem:expansions} makes these statements precise.

\begin{lemma} \label{lem:expansions}
Let $u=w-\phi(\alpha)$ and $u^+=w^+-\phi(\alpha^+)$.  Then
\begin{align}
  \alpha^+ &= \alpha-\alpha^3+O(|\alpha|^5+|\alpha|\|u\|),
	\label{eq:app-a} \\
  u^+ &= Bu+O(\alpha^4+\alpha^2\|u\|+\|u\|^2). \label{eq:app-e}
\end{align}
\end{lemma}
\begin{proof}
Conjugation by $D=\diag(1,-1,1)$ flips the signs of the $(1,2)$ and
$(2,3)$ entries.  It leaves $C$ and $\cA$ invariant, commutes
with~$\Pihat$, and maps $(\alpha,w)$ to $(-\alpha,w)$.  Hence
$\psi(-\alpha,w)=(-\alpha^+,w^+)$; that is, $\alpha^+$ is odd and $w^+$ is
even in~$\alpha$.

Since $\psi$ is real analytic, it follows that
$\alpha^+ = \alpha \, f(\alpha^2,w)$ and $w^+ = g(\alpha^2,w)$, with~$f$
and~$g$ real analytic near the origin.
Along $w=\phi(\alpha)$, we have $Z=Z_0(\alpha)$ with
\begin{equation} \label{eq:app-Z}
  Z_0(\alpha) = \begin{bmatrix}
	1 & \alpha & -\frac{3\alpha^2}{2} \\
	\alpha & 0 & -\alpha^3 \\
	-\frac{3\alpha^2}{2} & -\alpha^3 & \alpha^4-\frac{1}{2}
  \end{bmatrix}.
\end{equation}
The eigenpair equations therefore read
\[
  \lambda = 1 + \alpha v_2 - \frac{3}{2} \alpha^2 v_3, \qquad
  v_2 = \frac{\alpha-\alpha^3v_3}{\lambda}, \qquad
  v_3 = -\frac{3\alpha^2 + 2\alpha^3 v_2}{2\lambda + 1 - 2\alpha^4}.
\]
Since $\lambda$, $v_2$, and $v_3$ are analytic in $\alpha$ with
$(\lambda,v_2,v_3) = (1,0,0)$ at $\alpha=0$, substituting each estimate back
into these equations gives
\begin{gather*}
  \lambda = 1+\alpha^2+O(\alpha^4), \qquad
  v_2 = \alpha - \alpha^3 + O(\alpha^5), \qquad
  v_3 = -\alpha^2 + O(\alpha^4), \\
  w_1^+ = \tfrac{\lambda}{\|v\|^2} - 1 = O(\alpha^4), \qquad
  \alpha^+ = \alpha - \alpha^3 + O(\alpha^5), \qquad
  w_3^+ = -\alpha^2 + O(\alpha^4).
\end{gather*}
On the curve, $w_1=w_2=0$ and
$X_{22}+X_{13}=\alpha^2/(1+w_1)+w_3=0$; by the display, $w_1^+$ and
$X^+_{22}+X^+_{13}$ are $O(\alpha^4)$.  The dual update then gives
$w_2^+=O(\alpha^4)$ and $w_4^+=-2\alpha^2+O(\alpha^4)$, so
$w^+=\phi(\alpha)+O(\alpha^4)=\phi(\alpha^+)+O(\alpha^4)$.  Hence
\eqref{eq:app-a} and~\eqref{eq:app-e} hold for $u=0$.

At $\alpha=0$, we have $Z_{12}=Z_{23}=0$, hence $v_2=0$,
$\lambda=Z_{11}+Z_{13}v_3$, and $v_3=Z_{13}/(\lambda-Z_{33})$; linearizing
these relations and the update formulas in~$w$ gives $\D_w g(0,0)=B$.
Since~$g$ depends on~$\alpha$ only through~$\alpha^2$ and
$\phi(\alpha)=O(\alpha^2)$, it follows that
$\D_wg(\alpha^2,\phi(\alpha))=B+O(\alpha^2)$.  Expanding
$\alpha^+=\alpha f(\alpha^2,w)$ and $w^+=g(\alpha^2,w)$ in~$w$
around~$\phi(\alpha)$ now gives \eqref{eq:app-a} and~\eqref{eq:app-e}.
For the latter, note that $\phi(\alpha^+)$ depends on~$\alpha^+$ only
through $(\alpha^+)^2=\alpha^2f(\alpha^2,w)^2$, which changes by
$O(\alpha^2\|u\|)$ when $w$ moves away from~$\phi(\alpha)$.
\end{proof}

\paragraph{An invariant tube.}
We now specify the region~$\cQ$ on which $\psi$ coincides with PDHG and
which $\psi$ maps into itself: a tube of radius~$\alpha^3$ around the curve
$w=\phi(\alpha)$, $\alpha>0$.  With this radius, the deviation~$u$
perturbs~\eqref{eq:app-a} only by $O(\alpha^4)$, and the $O(\alpha^4)$
terms in~\eqref{eq:app-e} are absorbed by the contraction.

Since $\|B\| > 1$, we measure the deviation in a norm
in which $B$ is a contraction.
Fix $\theta\in(\sqrt3/2,1)$.  Since the spectral radius of $B$ is
$\sqrt3/2$, there are $c_0>0$ and $\theta_0\in(\sqrt3/2,\theta)$ with
$\|B^j\|\le c_0\theta_0^j$ for all $j$.  Hence
\[
  \tnorm{u}_B=\sum_{j\ge0}\theta^{-j}\|B^ju\|
\]
defines a norm on $\R[4]$. Moreover, there exists $c_B > 0$ such that
$\|u\|\le\tnorm{u}_B\le c_B\|u\|$, and 
$\tnorm{Bu}_B \le \theta \tnorm{u}_B$.
For $\oline{\alpha}>0$, define
\[
  \cQ = \big\{(\alpha,w) \;\big|\; 0 < \alpha < \oline{\alpha}, \
  \tnorm{w-\phi(\alpha)}_B \le \alpha^3 \big\}.
\]

\begin{lemma} \label{lem:tube}
For all sufficiently small $\oline{\alpha} > 0$,
\begin{enumerate}[label=(\alph*)]
\item for every $(\alpha,w)\in\cQ$, the eigenvalues of $Z$ are
\[
  \lambda_1 = 1+O(\alpha), \qquad \lambda_2=-\alpha^2+O(\alpha^3), \qquad
  \lambda_3 = -\tfrac{1}{2} + O(\alpha);
\]
in particular, $\lambda_2<0$, $\lambda_3<0$,
and $\proj_{\SV{3}{+}}(Z) = \Pihat(Z)$;
\item $\psi(\cQ) \subseteq \cQ$, and $\alpha^+=\alpha-\alpha^3+O(\alpha^4)$ on $\cQ$.
\end{enumerate}
\end{lemma}
\begin{proof}
(a) The matrix $Z_0(\alpha)$ in~\eqref{eq:app-Z} satisfies
$Z_0(\alpha) = Z_\star + O(\alpha)$, so, by Weyl's inequality, its largest
and smallest eigenvalues are $1 + O(\alpha)$ and $-\frac{1}{2} + O(\alpha)$.
Dividing $\det Z_0(\alpha) = \tfrac12 \alpha^2 + \alpha^6$ by their product
gives $\lambda_2(Z_0(\alpha)) = -\alpha^2 + O(\alpha^3)$.
For $(\alpha,w)\in\cQ$, $\|Z-Z_0(\alpha)\|=O(\|u\|)=O(\alpha^3)$,
which is small compared with~$\alpha^2$, so Weyl's inequality gives the
claims.

(b) On $\cQ$, $\|u\|\le\tnorm{u}_B\le\alpha^3$, so $\alpha\|u\|\le\alpha^4$, and \eqref{eq:app-a} gives
$\alpha^+=\alpha-\alpha^3+O(\alpha^4)$.  Hence, for small $\oline{\alpha}$,
$0<\alpha^+<\alpha$ and $(\alpha^+/\alpha)^3\ge1-6\alpha^2$.  On $\cQ$, the
$O(\cdot)$ term in \eqref{eq:app-e} is $O(\alpha^4)$,
so, for a constant~$c$,
\[
  \tnorm{u^+}_B\le\theta\tnorm{u}_B+c\,\alpha^4\le(\theta+c\alpha)\alpha^3.
\]
For small $\oline{\alpha}$, $\theta+c\alpha\le1-6\alpha^2\le(\alpha^+/\alpha)^3$,
and therefore $\tnorm{u^+}_B\le(\alpha^+)^3$; hence $\psi(\alpha,w)\in\cQ$.
\end{proof}

\begin{proof}[Proof of \cref{thm:sublinear}]
The initial point \eqref{eq:sublinear-init} corresponds to
$\alpha=\alpha_0$ and $w_0=(0, -3\alpha_0^4, -\alpha_0^2, -2\alpha_0^2)$.
Hence $u_0 = (0,-3\alpha_0^4,0,0)$ and
$\tnorm{u_0}_B \le 3c_B \alpha_0^4 \le \alpha_0^3$
once $\oline{\alpha} \le 1/(3c_B)$, so $(\alpha_0,w_0) \in \cQ$
whenever $\alpha_0 \in (0,\oline{\alpha})$.
By \cref{lem:tube} and induction, the iterates of $\psi$ remain in $\cQ$
and coincide with the PDHG iterates.
In particular, $\rank X_k=1$ and $0 < \alpha_{k+1} < \alpha_k$.  By \cref{lem:tube}(b),
$\alpha_{k+1}\le \alpha_k-\alpha_k^3/2$ for small $\oline{\alpha}$, so $\alpha_k\downarrow0$.  Then
\[
  \alpha_{k+1}^{-2} - \alpha_k^{-2}
  = \alpha_k^{-2} \big((1-\alpha_k^2+O(\alpha_k^3))^{-2}-1\big)
  = 2+O(\alpha_k) \to 2.
\]
Hence,
\[
  \frac{\alpha_k^{-2}}{k}
  = \frac{\alpha_0^{-2}}{k} + \frac1k \sum_{j=0}^{k-1}
  \big(\alpha_{j+1}^{-2} - \alpha_j^{-2}\big) \to 2,
\]
since the average of a convergent sequence has the same limit; that is,
$\sqrt{2k} \, \alpha_k \to 1$.  On $\cQ$, since $\|u\| \le \alpha^3$,
\begin{gather*}
  X_{11}=1+O(\alpha^3),\quad X_{12}=\alpha,\quad
  X_{13}=-\alpha^2+O(\alpha^3),\\
  X_{22}=\alpha^2+O(\alpha^5),\quad
  X_{23}=-\alpha^3+O(\alpha^4),\quad
  X_{33}=\alpha^4+O(\alpha^5),
\end{gather*}
and $y_1=O(\alpha^3)$, $y_2=-2\alpha^2+O(\alpha^3)$.  Hence
$\|X-E_{11}\|_\fro^2=2\alpha^2+O(\alpha^4)$ and
$\|y\|_2=2\alpha^2+O(\alpha^3)$, which give
$\sqrt k\|X_k-X_\star\|_\fro\to1$ and $k\|y_k-y_\star\|_2\to1$.
\end{proof}

\end{document}